\documentclass[12pt]{article}

\usepackage[utf8]{inputenc}
\usepackage[T1]{fontenc}
\usepackage[a4paper, left=2cm, right=2cm, top=2.5cm, bottom=2cm]{geometry}
\usepackage{amsfonts} %juste des polices
\usepackage{amsmath} %permets de mettre des formules
\usepackage{amssymb} %Symbole de math
\usepackage{amsthm} %Permets de définir des théorèmes
\usepackage[numbers]{natbib}%Pour la bibliographie
\usepackage{babel}% la langue
\usepackage{hyperref} % Met des liens hypertext
\usepackage{cleveref} %Permet de mettre des références à des figures et des équationx, à besoin de hyperref
\usepackage{color}%Met de la couleur dans le texte
\usepackage{dsfont} %des lettres avec deux barres
\usepackage{enumitem}% fait des listes à puces
\usepackage{graphicx}%permet de mettre des images
\usepackage{pifont}%des symboles bizarres
\usepackage{tabto}%fait des tabulations
\usepackage[all]{xy}%diagramme commutatif et plein d'autre chose
\usepackage{nomencl}%Devrait faire une nomenclature, pas réussi pour l'instant
\usepackage{wrapfig}%Met des images
\usepackage{authblk}

\usepackage{tikz}
\usepackage{tkz-euclide}
\usetikzlibrary{positioning, fit, calc, shapes, arrows, matrix, shapes.misc, intersections}

\theoremstyle{plain}% default
\newtheorem{thm}{Theorem}[section]
\newtheorem{lem}[thm]{Lemma}

\newtheorem*{cor}{Corollary}
\theoremstyle{definition}
\newtheorem{dfnt}{Definition}[section]

\theoremstyle{remark}
\newtheorem*{rmq}{Remark}

\crefname{thm}{theorem}{theorems}
\crefname{lem}{lemme}{lemmes}
\crefname{prop}{proposition}{propositions}
\crefname{cor}{corollaire}{corollaires}
\crefname{dfnt}{definition}{definitions}
\crefname{exmp}{exemple}{exemples}
\crefname{xca}{exercice}{exercices}
\crefname{rmq}{remarque}{remarques}
\crefname{note}{note}{notes}
\crefname{case}{case}{cases}
\crefname{hyp}{hypothèse}{hypothèse}

\newcommand{\HH}{\mathcal{H}}
\newcommand{\MM}{\mathcal{M}}
\newcommand{\SLR}{SL(2,\mathbb{R})}
\newcommand{\Imm}{\operatorname{Im}}
\newcommand{\Ree}{\operatorname{Re}}
\newcommand{\myparagraph}[1]{\paragraph{#1}\mbox{}\\}

\begin{document}

\title{Pairs of saddle connections of typical flat  surfaces on fixed affine orbifolds.}
\author[1]{Bonnafoux, Etienne \\ email: \href{mailto:etienne.bonnafoux@polytechnique.edu}{etienne.bonnafoux@polytechnique.edu}}
\affil[1]{Centre de mathématiques Laurent-Schwartz, Ecole polytechnique, 91128 Palaiseau Cedex, France}

\providecommand{\keywords}[1]{\textbf{\textit{Key words:}} #1}
\providecommand{\classification}[1]{\textbf{\textit{2020 Mathematics Subject Classification:}} #1}

\maketitle

\begin{abstract}
  We prove that the asymptotic number of pairs of saddle connections with length smaller than $L$ with bounded virtual area is quadratic (with power saving error term) for almost every translation surface with respect to any ergodic
  $SL(2,\mathbb{R})$-invariant measure. A key tool of the proof is that Siegel-Veech transforms of bounded functions with compact supports are in $L^{2+\kappa}$ for every $SL(2,\mathbb{R})$-invariant measure.
\end{abstract}


\classification{32G15, 30F30 (Primary);	28C10 (Secondary)}

\section{Introduction}

Given a translation surface $(X,\omega)$, that is a pair made of a Riemann surface $X$ and a non-zero holomorphic one form $\omega$, a saddle connection is defined as a flat geodesic between two zeros of $\omega$. All along this article we might write $\omega$ for a translation surface for succinctness.

To each saddle connection $\gamma$ is associated its holonomy vector $$
z_{\gamma}=\int_{\gamma} \omega \in \mathbb{C}.
$$
The length of $\gamma$ is then defined as $|z_{\gamma}|$.

We will count pairs of saddle connections with constraint length and \emph{virtual area} which is defined by the equation: $$
|z_{\beta} \wedge z_{\gamma}|=|\Ree(z_{\beta})\Imm(z_{\gamma})-\Imm(z_{\beta})\Ree(z_{\gamma})|
$$
for two saddle connections $\gamma$ and $\beta$.

For $\omega$ a translation surface, we call $\Lambda_\omega$ the set in $\mathbb{C}$ of all holonomy vectors of the saddle connections in $\omega$. This is a discrete locally finite set.

The moduli space $\Omega_g$ of compact genus $g$ area $1$ translation surfaces, is stratified by integer partitions of $2g-2$, and each stratum is divided in different connected components. For the rest of the paper we fix a genus $g$, an integer partition and one of its connected component $\HH$\cite{kontsevich2003connected}.

 On $\Omega_g'$ the moduli space of compact genus $g$ without area restriction, there is a natural measure $\mu_{MV}'$ called Masur-Veech measure which is locally the Lebesgue measure with respect to the period coordinate \cite{masurmeasure, Veechmeasure}. Then by coning, a measure $\mu_{MV}$, also named Masur-Veech measure, is defined on $\Omega_g$, that is for every $A$ measurable:
$$
\mu_{MV}(A)=\mu_{MV}'(\{ s \omega, 0\leq s \leq 1, \omega \in A\}).
$$

\subsection{History and prior results}

The saddle connections are keystones for studying the geometry of flat surfaces. In fact a result of Masur and Smillie \cite{masur_smilie} says that the compact part of $\Omega_g$ is made of surfaces with no saddle connection of length less than a given $\epsilon>0$.

Concerning the counting of these objects, Masur \cite{masur_1990} managed to bound the function $N(\omega, R) = \Lambda_{\omega} \cap B(0,R) $. More precisely, he showed that for every $\omega \in \HH$ , there are $c_1(\omega)$ and $c_2(\omega)$ such that $$
c_1(\omega) R^2 \leq \# N(\omega, R) \leq  c_2(\omega) R^2.
$$
Later on, Veech \cite{VeechQuad} found an $L^1$-quadratic asymptotic formula: namely, there is a $c>0$ such that for $\mu_{MV}$-almost every $\omega$,
$$
\underset{R \to \infty}{\lim} \int_{\HH} \left\vert \frac{N(\omega, R)}{R^2} - c \right\vert d \mu(\omega) =0
$$
This result was subsequently improved by Eskin and Masur \cite{eskin_masur_2001} who showed that, for every $\mu$ an ergodic $\SLR$-invariant measure, there is a $c>0$ such that for $\mu$-almost every $\omega \in \HH$
$$
\underset{R \to \infty}{\lim} \frac{N(\omega, R)}{R^2} =c.
$$
More recently Athreya, Fairchild and Masur \cite{athreya2022counting}  extended this result to a counting of pairs of saddle connections with a constraint on the virtual area.
More concretely, consider the following counting function $$
N_{A}(\omega, R)= \# \{(z, w) \in \left(\Lambda_{\omega} \cap B(0,R)\right)^2: |z\wedge w|\leq A, |w| \leq |z|\}.
$$

\begin{thm}
  \label{th1}
  There is a constant $c(A)$ such that for $\mu_{MV}$-almost every $\omega \in \HH$
  $$
  \underset{R \to \infty}{\lim} \frac{N_{A}(\omega, R)}{R^2} = c(A).
  $$
\end{thm}

The main goal of this paper is to generalize this result to other $\SLR$-invariant measures.

\subsection{General $\SLR$-invariant measures}

The group $\SLR$ acts on $\Omega_g$ and preserve $\mu_{MV}$ (which is ergodic with respect to its action). We will write $g_t$ for its diagonal action and $r_{\theta}$ for the rotation.

In general, a connected component of a stratum $\HH$ carries many $\SLR$-invariant measures. As it was famously shown by Eskin and Mirzakhani \cite{EskinMirzakhani}, such measures are supported on affine orbifolds, here denoted by $\MM$. Roughly speaking, this mean that in any point $p \in \MM$, there is a neighborhood $U$ such that $\MM \cap U$ is map by local period coordinates on a subspace defined by real linear equations.
%TODO c'est peut être vrai dans l'espace sans restriction d'aire

Our result is that theorem \ref{th1} also holds for this family of measures:
\begin{thm}
  \label{thm2}
  Given $\mu$ an ergodic $\SLR$-invariant measure, there is a $c_{\mu}(A)$ and a $\kappa > 0$ such that for $\mu$-almost every $\omega \in \HH$
  $$
   \frac{N_{A}(\omega, R)}{R^2} = c_{\mu}(A) +O_{\omega}(R^{-\kappa}).
  $$
  Furthermore if a sequence of probability $\SLR$-invariant ergodic measure $\mu_n$ weakly converge to a probability $\SLR$-invariant measure ergodic $\mu$, then for any $A>0$ $$
  c_{\mu_n}(A) \to c_{\mu}(A)
  $$
\end{thm}

The proof of this result occupies the rest of this paper. More precisely we begin by showing in Section \ref{sec2} and \ref{sec3} that $$\underset{R \to \infty}{\lim} \frac{N_{A}(\omega, R)}{R^2} = c_{\mu}(A).$$
Our argument will follow closely the proof of Athreya, Fairchild and Masur, with three significant changes, as indicated in Subsection \ref{recapproof} below.
Then in Section \ref{sec5} we will derive the power saving error term $O_{\omega}(R^{-\kappa})$, based on tools developed by Nevo, Rühr and Weiss \cite{nevo2020effective}.
Finally, continuity properties of the constants $c_{\mu}(A)$ are discussed in Section \ref{sec4}.

For our purpose, it is important to recall that there is a notion of independence for the saddle connection attached to each affine orbifold $\MM$.

 \begin{dfnt}
 Saddle connections $s_1 , \cdots , s_k$
 on a surface $\omega \in \MM$ are said to be $\MM$
 \emph{-independent}
  if their relative homology classes define linearly independent functionals (over $\mathbb{C}$) on the linear subspace $T_{\omega} \MM \subset T_{\omega} \HH \cong H^1(\omega,\Sigma;\mathbb{C})$ where $\Sigma$ stand for the set of $0$ of $\omega$.
 \end{dfnt}

\subsection{Siegel-Veech transform}

As remarked by Athreya, Fairchild and Masur in Section 1.4.3 of \cite{athreya2022counting} one key point to show theorem \ref{thm2} is to extend a result on the integrability of the Siegel-Veech transform of bounded compact functions which we recall now.

Let $B_c(X)$ be the space of bounded measurable functions with compact support on a space $X$. The Siegel-Veech transform for a function $f \in B_c(\mathbb{R})$ is defined as
$$
\hat{f}(\omega)= \sum_{z \in \Lambda_{\omega}} f(z).
$$
Sometimes for readability, we will also use the notation $f^{SV}$.

In this setting Athreya, Cheung and Masur \cite{athreya2017siegel} proved that:
\begin{thm}
  \label{SVtr}
  There is a $\kappa>0$ such that for every $f \in B_c(\mathbb{C})$, $\hat{f}\in L^{2+\kappa}(\HH,\mu_{MV})$.
\end{thm}

As it turns out, our goal in the next section, will be to show the following extension of Theorem \ref{SVtr}.
\begin{thm}
  \label{thm_SVtr2}
  For every $\SLR$-invariant measure $\mu$,
  there is a $\kappa>0$ such that for every $f \in B_c(\mathbb{C})$, $\hat{f}\in L^{2+\kappa}(\HH,\mu)$.
\end{thm}

As it is discussed in the beginning of Section 2.2 of \cite{athreya2022counting}, this theorem extends to function of two variables. For a function $g \in B_c(\mathbb{R}^2)$, the Siegel-Veech transform is defined as
$$
\hat{g}(\omega)= \sum_{z_1, z_2 \in \Lambda_{\omega}^2} g(z_1, z_2).
$$

An easy corollary of the previous theorem is
\begin{cor}
  For every $\mu$ $\SLR$-invariant measure, there is a $\kappa'>0$ such that for every $h \in B_c(\mathbb{R}^2)$, $\hat{h}\in L^{1+\kappa'}(\HH,\mu)$.
\end{cor}
\subsection{Review of the proof of Theorem \ref{th1}}
\label{recapproof}

As our argument towards Theorem \ref{thm2} follow closely the proof of Theorem \ref{th1}, we will give a quick landscape on it and comment on the points which need to be adapted.

To begin, Athreya, Fairchild and Masur focus on partial counting functions $$
N^*_A(\omega, R)= N_A(\omega, R)-N_A(\omega, R/2).
$$
These functions count the number of pairs of saddle connections such that their period coordinates are in
$$
D_A(R, R/2):= \left\{ (z, w) \in \mathbb{C}^2: R/2 < |z| < R
, |w|<|z|, |w \wedge z| \leq A  \right\}
$$
Estimates on these function can be extended to $N_A(\omega,e^t)$ using geometric series argument with a suitable control on a upper bound on $N^*_A(\omega, R)$ given by Proposition 3.2 in their article.

Then they describe a set $$
R_A(\mathcal{T}) := \left\{ (z, w) \in \mathbb{C}^2, 1/2 \leq \Imm(z) \leq 1, |\Ree(z)| \leq \Imm(z), |\Imm(w)|\leq \Imm(z), |w \wedge z| \leq A \right\}
$$
 called the fibered trapezoid, which have the property that its characteristic function $h_A$ satisfies
\begin{equation}
  \label{estim}
  |N^*_A(\omega, R) - \pi e^{2t}(A_t \hat{h}_A)(\omega)| = |m_t(\omega) + \sum_{i=1}^4 e^i_t(\omega)|
\end{equation}

where $A_t$ is the following averaging operator
$$
A_t(h)(p)= \frac{1}{2 \pi} \int_0^{2 \pi} h(g_t r_{\theta}p) d\theta
$$
and $m_t$ is called \emph{main term} and the $e^i_t$ \emph{error terms}.

These functions are defined as the difference of the two terms in the left hand side of equation \eqref{estim} estimated on various loci. The definitions of these loci are the following

\begin{itemize}
  \item \emph{Main part} $$
  M_t = \left\{ (z, w)\in D_A\left( \sqrt{\frac{cosh(2t)}{2}}, e^t \right) : |w|< |z|(1+e^{-4t})^{-1/2} \right\},
  $$
  \item \emph{Bottom of the trapezoid }$$
  E^1_t = D_A \left( e^t/2, \sqrt{\frac{cosh(2t)}{2}} \right),
  $$
  \item \emph{The norm of $w$ is greater than $|z|(1+e^{-4t})^{-1/2}$}$$
  E^2_t = \left\{ (z, w)\in D_A\left( \sqrt{\frac{cosh(2t)}{2}}, e^t \right) :|w|> |z|(1+e^{-4t})^{-1/2} \right\},
  $$
  \item \emph{Top of the trapezoid }$$
  E^3_t = \left\{ (z, w)\in \mathbb{C}^2 : (A_t h_A)(z, w)>0,|z|>e^t \right\},
  $$
  \item \emph{The averaging operator is positive but} $(z,w) \notin D_A(e^t/2, e^t)$
  $$
  E^4_t = \left\{ (z, w)\in \mathbb{C}^2 : (A_t h_A)(z, w)>0, e^t/2<|z|<e^t,|w|>|z| \right\}.
  $$
\end{itemize}
After that the error functions are defined as $$
e^i_t(\omega)= (\chi_{E^i_t}\cdot (\chi_{D_A(e^t/2, e^t)}-\pi e^{2t}(A_t h_A)))^{SV}(\omega)
$$
  and
  $$m_t(\omega)=(\chi_{M_t}\cdot (\chi_{D_A(e^t/2, e^t)}-\pi e^{2t}(A_t h_A)))^{SV}(\omega).$$
where $\chi_{\cdot}$ are characteristic functions (this notation will be used all along this article).

To control the limit of $(A_t \hat{h}_A)$ as $t \to \infty$, they use Nevo ergodic theorem \cite{NevoErgo} stating that
\begin{thm}
  \label{thm_Nevo}
  Suppose $\mu$ is an ergodic $\SLR$-invariant probability measure on $\HH$. Assume $f \in L^{1+\kappa}(\HH,\mu)$ for some $\kappa>0$, and that $f$ is $K$-finite, that is, $f_{\theta}(\omega) := f(r_{\theta}\omega)$ the span of the functions $\{ f_{\theta} : \theta \in [0,2 \pi[ \}$ is finite-dimensional. Let $\eta \in C_c(\mathbb{R})$
  be a continuous non negative bump function with compact support and of unit integral. Then for $\mu$-almost every $\omega \in \HH$,
  $$
\underset{t \to \infty}{\lim} \int_{-\infty}^{\infty} \eta(t-s)(A_s f)(\omega)ds = \int_{\HH}f d\mu
  $$
\end{thm}

Using this theorem they get a result on the limit of the averaging operator.

\begin{thm}[Theorem 2.2 of \cite{athreya2022counting}]
  \label{thm_LimAver}
  For $\phi \in C_c(\mathbb{C}^2)$, for $\mu_{MV}$-almost every $\omega \in \HH$, the circle average of $\hat{\phi}$ converge
  $$
  \underset{t \to \infty}{\lim} A_t \hat{\phi}(\omega)= \int_{\HH} \hat{\phi} d \mu_{MV}
  $$
\end{thm}

To adapt this theorem to our case of an arbitrary $\SLR$-invariant ergodic measure $\mu$, we need to show that there is a $\kappa>0$ such that for any $f \in B_c(X)$, $f \in L^{2+2\kappa}(\HH,\mu)$ which will imply that for any $h \in B_c(\mathbb{C}^2)$, $\hat{h} \in L^{1+\kappa}(\HH,\mu)$.
For the Masur-Veech measure, this statement is the main result of \cite{athreya2017siegel}. In our setting, this is the main content of Section \ref{sec2}.

Once this is done, taking a bump function $g_\epsilon$ (whose existence is established in Lemma 3.4 of \cite{athreya2022counting}) such that $$
|A_t \hat{h}_A(\omega)-A_t \hat{g}_{\epsilon}(\omega)| \leq \epsilon
$$
 and
 $$
 | \int_{\HH} \hat{g}_{\epsilon}-\hat{h}_A d \mu| \leq \epsilon
 $$
we get that $$
|A_t \hat{h}_A(\omega)-\int_{\HH}\hat{h}_A d\mu| \leq |A_t \hat{h}_A(\omega)-A_t \hat{g}_{\epsilon}(\omega)| + |A_t \hat{g}_{\epsilon}(\omega) - \int_{\HH} \hat{g}_{\epsilon} d \mu| + | \int_{\HH} \hat{g}_{\epsilon}-\hat{h}_A d\mu| \leq 3 \epsilon
$$
for $t$ big enough, using the extension for $\mu$ of Theorem \ref{thm_LimAver} on $A_t \hat{g}_\epsilon$ for the second term.

It remains to prove that for almost every $\omega$,
\begin{equation}
  \label{estimMt}
  |m_t(\omega)| = o(e^{2t})
\end{equation}

and for $i=1, 2, 3, 4$
\begin{equation}
  \label{estimEt}
  |e^i_t(\omega)|=o(e^{2t}).
\end{equation}

For this sake, in the case of the Masur-Veech measure, Athreya, Fairchild and Masur started by estimating the volume of the set with two non-homologous small saddle connections (cf. Lemma 5.2 of \cite{athreya2022counting}). In our current setting this lemma has an equivalent form thanks to Dozier \cite{dozier2020measure}.

Next, Athreya, Fairchild and Masur have two lemmas (cf. Lemmas 5.3 and 5.4 of \cite{athreya2022counting}) bounding the function  $N(\omega, L)$.
 As they show in their paper, these lemmas essentially follow paragraph 3.6.2 and 3.6.3 of \cite{athreya2017siegel} and \cite{eskin_masur_2001}.
Since they are stated for every $\omega \in \HH$, we don't need to modify them here.

%Afterwards, Athreya, Fairchild and Masur bound (cf. Lemma 5.5 of \cite{athreya2022counting}) the Masur-Veech integral of a counting function in the thin part of the moduli space (where some saddle connections are small).
%This uses only the estimate of Lemma 5.2 and hence extends to any $\SLR$-invariant measure (because as it was previously said Dozier's result \cite{dozier2020measure} allows us to extend the result of Lemma 5.2).

  Afterwards, Athreya, Fairchild and Masur bound (cf. Lemma 5.5 of \cite{athreya2022counting}) the Masur-Veech integral of a counting function in the thin part of the moduli space (where some saddle connections are small).
  This mostly uses the estimate of Lemma 5.2 (which as it was previously said is extended by Dozier's result \cite{dozier2020measure} and can be used to have similar result of Lemma 5.2) except a couple of technical terms should be adapted to be extended to any $\SLR$-invariant measure.

 Furthermore, Lemmas 5.6 and 5.7 of \cite{athreya2022counting} are technical lemmas of a geometric nature,
  controlling the orbit of a point in $\mathbb{C}$ under geodesic flow and rotation: this means that we can still use it.

Moreover, they need to control the Masur-Veech volume of different loci on the thick part of the stratum related to the error terms (cf. Lemma 5.9 of \cite{athreya2022counting}).
In the case of an arbitrary $\SLR$-invariant ergodic measure, we shall need to add extra arguments. This will be done in Section \ref{sec3}.

 In the end, using all these facts, we can follow in a straightforward way the same computation of Athreya, Fairchild and Masur in Section 5.9 of \cite{athreya2022counting} to get the bounds on the main term \eqref{estimMt} and the error terms \eqref{estimEt} and \emph{a fortiori } combining the different limits the desired result.

In summary, our task for proving the quadratic asymptotics in Theorem \ref{thm2} is reduced to show the integrability of the Siegel-Veech transform with respect to any $\SLR$-invariant ergodic measure (cf. Theorem \ref{thm_SVtr2}) which will be done in Section \ref{sec2} and two adapt three lemmas (Lemma 5.2, 5.5 and 5.9 of \cite{athreya2022counting}) which will be the content of Section \ref{sec3}.

\subsection{Compactification of the moduli space}

  A crucial tool to show Theorem \ref{thm_SVtr2} is the multi-scale compactification of strata introduced by Bainbridge, Chen, Grushevsky and Möller \cite{Bainbridge_2018}.

  Indeed, Dozier uses it to bound the volume of a set with a family of $\MM$-independent small saddle connections.
  Similarly, we want an estimate on the volume of the set where two saddle connections are small and in some precise position.
   Namely, one is a circumference curve of a cylinder and the other one cross the same cylinder.
    This set appears in the proof of the integrability of the Siegel-Veech transform, and in some sens, the control of his measure is one of the novelty of this paper.

  Let us recall quickly the compactification process.
  A multi-scale differential is the data of
\begin{itemize}
  \item A nodal Riemann surface $M$,
  \item A graph, where each vertex correspond to a component of the surface $M$, and the edge to a node between the surfaces,
  \item Half-edge recording in which component the $0$ of the differential are,
  \item A level function $l$ assigning a non positive integer to each vertex. We take it to be surjective in $[-N, 0]\cap \mathbb{Z}$ for some $N$. The union of component with the same level, that is $l^{-1}(i)$ for a given $i$ is called a \emph{level subsurface} and denote by $X^{(i)}$,
  \item A positive integer $b_e$ to each vertical edge which record the cone angle,
  \item A collection of meromorphic differentials on each component of the surface which are consistent with the previous data (we won't give here the meaning of this term, but we refer the reader to \cite{Bainbridge_2018} for full definition),
  \item And a prong-matching record, giving the identification of the horizontal direction between two sides of a node.
\end{itemize}

The compactified stratum is denoted $\bar{\HH}$. On it, special neighborhoods are handy for computation. They are described by the following definition and theorem.

\begin{dfnt}
  A connected, open subset $Q \subset \bar{\HH}$ is said to be a \emph{period coordinate chart} if it admits an injective map to $\mathbb{C}^n$ that is locally linear with respect to the period coordinates.
\end{dfnt}

Fix $\bar{X}$ in the boundary, there is a complex-analytic coordinate in which we can define a good neighborhood.
\begin{itemize}
  \item In each component, there are \emph{moduli parameters} $s_i \in \mathbb{C}$ that we take small, that is $0 < |s_i| < \epsilon$,
  \item For each node, there is a \emph{smoothing parameter} $t \in \mathbb{C}$. We also take them with a restriction on the norm $0 < |t| < \epsilon$,
  \item If $t$ is a horizontal node parameter, then we consider a restriction of the smoothing parameter $$
arg(t) \in (\alpha, \alpha + \pi /4)
  $$
  where we choose finitely many $\alpha$ such that the family of such intervals cover all the circle,
  \item If $t_i$ is a scaling parameter for level $i$, with associate integer $a_i$, we find connected interval condition on $arg(t_i)$ such that $arg(t_i^{a_i})$ satisfies the same condition as in the previous point.
\end{itemize}

Dozier show in Lemma 3.6 of \cite{dozier2020measure} that:

\begin{thm}
  \label{NeiPerCa}
  The previously defined neighborhoods are period coordinate charts, for $\epsilon$ small enough.
\end{thm}

\subsection{Ordering of level subsurfaces}

The ordering given by the level function might not be a total order. To structure the surface, we should make other choices.

Let $\mathcal{S}$ be the set made of level subsurfaces and degenerating cylinders. Dozier defined a function $size_{\omega}(\cdot):\mathcal{S} \to \mathbb{R}$ for every $\omega \in V$ with $V$ a neighborhood as in theorem \ref{NeiPerCa}.

Then, if $\mathcal{O}$ is the set of ordering on $\mathcal{S}$ that restricts to the ordering given by the level function and have the property that if $C$ is a degenerating cylinder and $X^i$ is the level subsurface at which the circumference curve of $C$ lies, then $C \succ X^i$.

 We say that $\succ \in \mathcal{O}$ is \emph{consistent} with $X \in V$ if $$
size_{\omega}(Y_1) \geq size_{\omega}(Y_2) \implies Y_1 \succ Y_2.
$$

Finally, we call a degenerating cylinder $\succ$\emph{-wide} if it is $\succ$-greater than any level subsurface.

\subsection{Delaunay triangulation}

One last tool that we will use in Section \ref{sec3} is Delaunay triangulation. Let's consider the \emph{Voronoi decomposition} of a translation surface with each cell centered around singular point. Then the dual of this decomposition is the \emph{Delaunay decomposition}. Further decompositions are called \emph{Delaunay triangulation}. Refer to \cite{masur_smilie} for complete definition.

One key point of these triangulations is the so-called efficiency property
\begin{thm}
  \label{effpath}
  Let $S$ be a translation surface and $T$ a Delaunay triangulation.
  Then every saddle connection $\beta$ is homotopically equivalent to a path $P(\beta)$ in the Delaunay triangulation whose length satisfies $$
  |P(\beta)| \leq \sqrt{10}|\beta|.
  $$
  \begin{proof}
    See Lemma 2.3 of \cite{athreya2017siegel}.
  \end{proof}
\end{thm}

\subsection{Acknowledgements}
 I would like to thank Benjamin Dozier for his crucial remarks which greatly improve this paper. I am thankful to Jon Chaika for pointing me out the result of Dozier on convergence of Siegel-Veech constants. I am also grateful to Jayadev Athreya, Samantha Fairchild and Howard Masur for inspiring discussion at CIRM during the Combinatorics, Dynamics and Geometry on Moduli Spaces conference. Finally, I would like to thank Carlos Matheus for introducing me to the subject and advising me all along.

\section{Siegel-Veech transforms are in $L^{2+2\kappa}(\HH,\mu)$}
\label{sec2}

In this section, we now prove Theorem \ref{thm_SVtr2}.

We use the same strategy as in \cite{athreya2017siegel} by starting with the characteristic function of the disc of radius a small $\epsilon_0$ centered at $0$.

\begin{thm}
  \label{thm_SV_epsilon}
  Let $\epsilon_0 > 0$ and let $f$ be the characteristic function of the disc of radius $\epsilon_0$ centered at $0$. Then there is a $\kappa>0$ such that $\hat{f}$ is in $L^{2+2\kappa}(\HH,\mu)$.
\end{thm}

\begin{rmq}
  $\epsilon_0$ should be taken small enough later, in the last paragraph.

\end{rmq}

 We discuss the integrability of this function on several loci of $\HH$. For all of them, but the last one, the arguments are the same as in \cite{athreya2017siegel}. We will briefly remind them to concentrate on the last one.

\myparagraph{Thick part}
On the locus with no saddle connection with length smaller than $\epsilon_0$, we have $\hat{f}=0$.

\myparagraph{No short loops}
On the set where there is a saddle connection of length less than $\epsilon_0$, but no homotopic nontrivial closed curves of length less than $\epsilon_0$,
following paragraph 3.4 of \cite{athreya2017siegel}, $\hat{f}$ is bounded on this set. As the measure $\mu$ is of finite volume, this concludes this case.

\myparagraph{Short loops}
\label{shortloop}
On the case where there are short loops of length smaller than $\epsilon_0$, we will subdivide into four cases $\Omega_0,\Omega_1,\Omega_2,\Omega_3$.

On each, we want to show that $$
\sum_{k \geq 1} \mu(\Omega_i(k)) < + \infty
$$
where $\Omega_i(k):=\{\hat{f} \geq k^{1/q} \} \cap \Omega_i$, for a $q>2$ that we will fix latter.

For the rest of this section we will call $\gamma$ the smallest saddle connection and $|\epsilon|$ the length of the smallest saddle connection which is non parallel with $\gamma$. Note that in the case of two saddle connections  non-parallelism implies $\MM$-independence as explain in $\S$ 1.4 of the article of Dozier \cite{dozier2020measure}.

Theorem 5.1 of \cite{eskin_masur_2001} ensures that if $N$ is the dimension of relative homology, for every $\delta< \frac{1}{N}$ there is a $C$ such that
$$
\hat{f}(\omega) \leq \frac{C}{|\gamma|^{1+\delta}}.
 $$

We can deduce that if $\hat{f}\geq k^{1/q}$ there is a $c$ such that $$
|\gamma| \leq c k^{-\frac{1}{q(1+\delta)}}.
$$

Choose for the rest of the section, any $\delta$ and $p$ such that
$$
0<\delta<p< \frac{1}{N}.
$$
We will fix them later.

\myparagraph{The shortest non-parallel saddle connection is shorter that a power of the shortest}
Let
$$
\Omega_0(k)=\{(X,\omega) \in \HH, \hat{f}(\omega) \geq k^{1/q},|\epsilon| \leq |\gamma|^p\}.
$$
The regularity property of the measure $\mu$ proved by Dozier \cite{dozier2020measure}, knowing non-parallelism and hence the $\MM$-independence of $(\gamma,\epsilon)$, ensures that $$
\mu(\Omega_0(k))=O(|\gamma|^{2+2p})=O(k^{-\frac{2(1+p)}{q(1+\delta)}}),
$$
which is summable if $q<2\frac{1+p}{1+\delta}$

\myparagraph{The shortest non-parallel saddle connection is longer that a power of the shortest and the shortest is not a cylinder curve}
Let
$$
\Omega_1(k)=\{(X,\omega) \in \HH, \hat{f}(\omega) \geq k^{1/q},|\epsilon| \geq |\gamma|^p , \gamma \textrm{ is not on the boundary of a cylinder}\}
$$

Let's recall a lemma from Athreya, Cheung and Masur \cite{athreya2017siegel}.
\begin{lem}[Lemma 3.2 of \cite{athreya2017siegel}]
  Suppose $\gamma$ is the shortest saddle connection on $(X,\omega)$. Let $\beta$ be a saddle connection (with an orientation) such that the path $P(\beta)$ in the Delaunay triangulation given by Theorem \ref{effpath} follows edges parallel to $\gamma$ more than $2M+1$ times, where $M$ is the total number of triangles in the Delaunay triangulation. Then there is a cylinder $C$ with $\gamma$ on its boundary and $\beta$ crosses $C$.
\end{lem}

As in this case $\gamma$ is not on the boundary of a cylinder, every saddle connection $\beta$ such that the path $P(\beta)$ follow $2M$ times a edge parallel with $\gamma$ must follow an edge non parallel with $\gamma$ after that, and hence of length at least $| \epsilon |$. Thus any saddle connection with length less that $\epsilon_0$ can be decomposed as $O(|\epsilon|^{-1})$ edges of the triangulation and hence decomposed in a basis of $H_1(X,\omega,\Sigma)$ with coefficient that are $O(|\epsilon|^{-1})$. The number of saddle connection smaller than $\epsilon_0$ is then $O(|\epsilon|^{-N})$ where
$N$ is the dimension of $H_1(X,\omega,\Sigma)$.

To summary
$$
\hat{f}(X,\omega) = O(|\epsilon|^{-N}) = O(|\gamma|^{-Np}).
$$
The result recalled at the beginning of subparagraph "short loops" gives that,  if $\hat{f}$ has to be bigger than $k^{1/q}$, it is needed that
$$
|\gamma|=O(k^{\frac{-1}{qNp}}).
$$
Thus, using the regularity of the measure $\mu$,
$$
\mu(\Omega_1(k))=O(|\gamma|^2)=O(k^{\frac{-2}{qNp}}),
$$
which is summable if $q < \frac{2}{Np}$.

\myparagraph{The shortest non-parallel saddle connection is longer that a power of the shortest and the shortest is a cylinder curve and the height of the cylinder is at least $\epsilon_0$}
Let
$$
\Omega_2(k)=\{(X,\omega) \in \HH, \hat{f}(\omega) \geq k^{1/q},|\epsilon| \geq |\gamma|^p , \gamma \textrm{ is on the boundary of a cylinder of height } \geq \epsilon_0\}
$$

Here the same lemma as in the previous paragraph is applied, and we get that a curve that follow $2M$ times edges parallel to $\gamma$ should after cross the cylinder whose height is $\epsilon_0$. As the number of parallel edges is bounded, so is $\hat{f}$.
Then $\Omega_2(k)$ is empty for $k$ big enough.

\myparagraph{The shortest non-parallel saddle connection is longer that a power of the shortest and the shortest is a cylinder curve and the height of the cylinder is at most $\epsilon_0$}
Let
$$
\Omega_3(k)=\{(X,\omega) \in \HH, \hat{f}(\omega) \geq k^{1/q},|\epsilon| \geq |\gamma|^p , \gamma \textrm{ is on the boundary of a cylinder of height } \leq \epsilon_0\}.
$$

For every point $X \in \bar{\HH}$ choose a neighborhood which is a period coordinate chart as in theorem \ref{NeiPerCa}. As the space is compact, we can extract a finite covering. This family separates into two kinds, the first one which are contained in the open locus of the compactification $\HH$, and the others which have a point on the boundaries $\bar{\HH}-\HH$.

We fix $\epsilon_0$ small enough, such that $\Omega_3(k)$ is included in a finite union of neighborhoods of the second type for every $k$.

Let $V$ be one of these neighborhoods.
In this neighborhood we pick an ordering $\succ$ of the subsurface. There is only a finite number of choice of neighborhood and of ordering.

We are interested in the intersection of $V$ with $\Omega_3(k)$ and the set of surfaces consistent with $\succ$. Let's call this set $B^{\succ}_{\Omega_3(k), V}$.

Let $\eta$ be a saddle connection included in the cylinder $C$, joining two zeros on the boundaries of $C$ and crossing $\gamma$ only once.
$\eta$ and $\gamma$ are $\MM$-independent since they are not parallel.

We add $\alpha_1,\cdots,\alpha_{k'}$, other saddle connections with representatives leaving at level at most $X^{(0)}$ such that $\gamma,\eta,\alpha_1,\cdots,\alpha_{k'}$ generate $H_{X^0 \succ}$ the subspace of elements of $H_1(X,\Sigma;\mathbb{C})$ that lie at level $X^0 \succ$. They do not cross any $\succ$-wide cylinder.

Then for each $\succ$-wide cylinder, one by one, in the order given by $\succ$, if the circumference curve is independent of the previous basis we add a curve joining two zeros and crossing no other $\succ$-wide cylinder (this is always doable). These saddle connections are called $\alpha_{k'+1},\cdots,\alpha_{m'}$.

Then we can extract an $\MM$-independent basis by removing, in the order of apparition in $\gamma,\eta,\alpha_1,\cdots,\alpha_{k'},\alpha_{k'+1},\cdots,\alpha_{m'}$, the $\MM$-dependent saddle connections.
We get a new basis $\gamma,\eta,\beta_1,\cdots,\beta_{k},\beta_{k+1},\cdots,\beta_{m}.$

Next we bound the period coordinates of this basis. To do this we recall two lemmas from Dozier \cite{dozier2020measure}.
\begin{lem}
  \label{doz1}
  For $\alpha$ a relative homology class defined on the surfaces in $V$, there exists a constant $C>0$ with the following property. Fix $X \in V$ of area $1$ such that $\succ$ is consistent with $X$. Let $Y$ be the level of $\alpha$ with respect to $\succ$. Then $$
  |\alpha(X)|\leq C size_X(Y).
  $$
\end{lem}
\begin{lem}
  \label{doz2}
  Let $\beta$ be a relative homology class defined on the surface in $V$ with the following properties:
  \begin{itemize}
    \item $\beta$ has a representative that crosses exactly one closed curve $\alpha$ that is the core curve of a degenerating cylinder. We let $Y$ be the level subsurface containing $\alpha$.
    \item $\beta$ has a representative such that the level subsurfaces which the representative intersects all lie at or below the level of $Y$.
  \end{itemize}
  Then there exists a family of rectangle $\mathcal{R}(z) \subset \mathbb{C}$, of area bounded above by some $R$ (depending on $V$ and $\beta$) with the following property. For any $X \in V$ of area $1$, we have,$$
\beta(X) \in \mathcal{R}(\alpha(X)).
  $$
  Furthermore, the rectangles have the property that $s\mathcal{R}(z) \subset \mathcal{R}(s z)$ for any $z \in \mathbb{C}$ and $0<s\leq 1$.
\end{lem}

Using this two lemmas and prior information on $\gamma$ and $\eta$, we can state that:

\begin{itemize}
    \item The period coordinates of $\gamma$ live in a disk of radius $ck^{-\frac{1}{q(1+\delta)}}$,
    \item The period coordinates of $\eta$ live in a rectangle centered at the origin with one side of length $|\gamma|$ and one side $\epsilon_0$, we call $R(z_{\gamma},\epsilon_0)$ this rectangle,
    \item All $\beta_i$ with $1 \leq i \leq  k$, with lemma \ref{doz1} are bound by $$
    |\beta_i(X')| \leq C size_{X^{0}}(X') \leq K
    $$
    because the size of the top subsurface is bounded in $V$,
    \item All $\beta_i$ with $k+1 \leq i \leq m$ are in $\mathcal{R}(\gamma_i(X'))$ with $\gamma_i$ the circumference curves of the $\succ$-wide cylinders they cross, and $\mathcal{R}(\cdot)$ the rectangles given by lemma \ref{doz2}. There are linear functions $f_i$ such that $\gamma_i= f_i(\gamma,\eta,\beta_1,\cdots,\beta_k)$ for all $k<i\leq m$.
     Indeed $\gamma_i$ belong to a sub-surface $X^i$ below the level $X^0$ and $H_{X^0 \succ }$ is generated by $\gamma,\eta,\beta_1,\cdots,\beta_k$. We call $\mathcal{R}_i$ the rectangle  $\mathcal{R}(f_i(\gamma,\eta,\beta_1,\cdots,\beta_k))$ .
\end{itemize}

We can now integrate using Fubini. To shorten the notation we call the period coordinate of $\beta_i$ $z_i := z_{\beta_i}$.

\begin{align*}
  \mu(\{ B^{\succ}_{\Omega_3(k), V} \}) & =  \mu'(\{ s X : 0 \leq s \leq 1, X \in B^{\succ}_{\Omega_3(k), V}\})
  \\
   & =  Leb(\{ s(z_{\gamma}, z_{\eta}, z_1,\cdots, z_k) \in \mathbb{C}^{k+2}, 0 \leq s \leq 1, z_{\gamma} \in B(0,ck^{-\frac{1}{q(1+\delta)}}),
  \\
 &
 z_{\eta} \in R(z_{\gamma},\epsilon_0), z_i \in B(0, K), 1 \leq i \leq m, z_j \in \mathcal{R}_j, k+1 \leq j \leq m \})
 \\
 &
 \\
 & \leq Leb(\{s(z_{\gamma}, z_{\eta}, z_1,\cdots, z_k) \in \mathbb{C}^{k+2}, 0 \leq s \leq 1, s z_{\gamma} \in B(0,s ck^{-\frac{1}{q(1+\delta)}}),
\\
&
s z_{\eta} \in R(sz_{\gamma},s\epsilon_0), s z_i \in B(0, K), 1 \leq i \leq m, s z_j \in s \mathcal{R}_j, k+1 \leq j \leq m \})
\\
&
\\
& \leq Leb(\{(z_{\gamma}, z_{\eta}, z_1,\cdots z_k) \in \mathbb{C}^{k+2}, z_{\gamma} \in B(0, ck^{-\frac{1}{q(1+\delta)}}),
\\
&
s z_{\eta} \in R(z_{\gamma},\epsilon_0),  z_i \in B(0, K), 1 \leq i \leq m,  z_j \in  \mathcal{R}_j, k+1 \leq j \leq m \})
\\
&
\\
& =
\underset{B(0,ck^{-\frac{1}{q(1+\delta)}})}{\int}
\underset{R(z_{\gamma},\epsilon_0)}{\int}
\underset{B(0, K)}{\int}
\underset{\mathcal{R}_j}{\int}
dz_m \wedge d\bar{z}_m \cdots dz_k \wedge d\bar{z}_k \cdots dz_{\eta} \wedge d\bar{z}_{\eta} dz_{\gamma} \wedge d\bar{z}_{\gamma}
 \\
 & =  O(k^{-\frac{3}{q(1+\delta)}})
\end{align*}

Finally summing over all choices of ordering and of neighborhood, we get that
$$
\mu(\Omega_3(k))=O(k^{-\frac{3}{q(1+\delta)}})
$$
which is summable if $q<\frac{3}{1+\delta}$.

\subsection{End of the proof}
In order to have the required integrability, we should now adjust the variables according to three conditions, namely:
\begin{itemize}
    \item $q<\frac{2}{Np}$
    \item $q<2 \frac{1+p}{1+\delta}$
    \item $q<\frac{3}{1+\delta}$
\end{itemize}
So we fix $p$ and $\delta$ as $$
\delta=\frac{1}{16N}<p=\frac{1}{8N}<\frac{1}{N}
$$
This yields the condition $q< 2+ C(N)$ with $C(N)$ some positive constant. Overall choosing $q$ between $2$ and $2+ C(N)$ we have shown the $2+\kappa$-integrability of the Siegel-Veech transform for a $\kappa>0$.

\begin{cor}
  For every $R>0$ the characteristic function of $B(0, R)$ is in $L^{2+\kappa}(\HH,\mu)$.
\end{cor}
\begin{proof}
  The proof is in \cite{athreya2017siegel} at theorem 3.3. We reproduce it here for completeness.

  Cover the disk of radius $R$ with sectors of angle $\frac{\epsilon_0^2}{R^2}$. By linearity is enough to show that for each sector, the characteristic function $f$ has its Siegel-Veech transform $\hat{f} \in L^{2+\kappa}(\HH,\mu)$.

  Let $(g_t)_{t \in \mathbb{R}}$ be the diagonal action and $(r_{\theta})_{\theta \in \mathbb{R}}$ the rotational action of $\SLR$ on the strata of the moduli space.
  Moreover, let $t_0 = \log (\frac{R}{\epsilon_0})$ and $\theta_0$ be the center angle of the sector.
   For any translation surface, if we apply $g_{t_0}r_{-\theta_0}$, any saddle connection with period coordinate in the sector will have in the new geometry a length less than $\epsilon_0$. Calling $h$ the characteristic function of the disk of radius $\epsilon_0$, since this two flow are measure preserving we have $$
\int \hat{f}^{2+\kappa} d \mu < \int \hat{h}^{2+\kappa} d \mu < + \infty.
  $$
\end{proof}
%TODO revoir ça peut être prendre epsilon_0'
The proof of Theorem \ref{thm_SVtr2} is now an easy consequence of the previous corollary.
\begin{proof}[Proof of Theorem 1.4]

  Indeed, take any $f \in B_c(\mathbb{R}^2)$. By assumption there is a $R$
   such that $|f| \leq \| f \|_{\infty} \chi_{B(0, R)}$,
    with $\chi_{B(0, R)}$ the characteristic function on the ball $B(0, R)$. Positivity of the Siegel-Veech transform yields the result.
\end{proof}

A corollary of Theorem \ref{thm_SVtr2} is that we can extend Theorem \ref{thm_LimAver} to any ergodic $\SLR$-invariant measure $\mu$.
\begin{thm}
  \label{cor_LimAver}
  For any ergodic $\SLR$-invariant measure $\mu$ and for any $\phi \in C_c(\mathbb{C}^2)$, for $\mu$-almost every $\omega \in \HH$, the circle average of $\hat{\phi}$ converge
  $$
  \underset{t \to \infty}{\lim} A_t \hat{\phi}(\omega)= \int_{\HH} \hat{\phi} d \mu
  $$
\end{thm}
\begin{proof}
  The proof is identical to the proof of Proposition 2.2 of \cite{athreya2022counting}, but using the that $\phi \in L^{1+\kappa}(\HH,\mu)$ with a general $\mu$.
\end{proof}

\section{Bounds on the volume of some sets}
\label{sec3}

To complete the discussion of the quadratic asymptotic for $N_A(\omega,R)$, we need to revisit the proof of three technical lemmas, namely Lemma 5.2, 5.5 and 5.9 of \cite{athreya2022counting}.

The first one is not hard to extend thanks to the work of Dozier \cite{dozier2020measure}.
In our setting, it states that if $\mu$ is a $\SLR$-invariant measure then:

\begin{lem}
  \label{lem_Dozier}
    For all $\epsilon,\kappa > 0$, the $\mu$-volume of the set $V_1(\epsilon,\kappa) \subset H$ of $\omega$ which have a saddle connection of length at most $\epsilon$, and a non-homologous saddle connection with length at most $\kappa$ is $O(\epsilon^2 \kappa^2)$
\end{lem}
\begin{proof}
The analog of this lemma for all $\SLR$-invariant measure is precisely the main result of Dozier, that is Theorem 1.1 of \cite{dozier2020measure}.
 \end{proof}

 The second one bounds the integral of the Siegel-Veech transform of a function with compact support on the thin part of the moduli space.

 \begin{lem}
   \label{lem_IntThinpart}
  Let $N$ be the maximal number of edge in a Delaunay triangulation and choose $\delta$ such that $\delta < 1/2N$. Let $\HH_\epsilon$ be the locus of moduli space where the shortest saddle connection has length less than $\epsilon_1$.
  Let $\psi$ be the Siegel-Veech transform of the characteristic function of the ball $B(0,L_0) \subset \mathbb{C}^2$ then $$
  \int_{\HH_{\epsilon_1}} \psi d \mu = O(\epsilon_1^{1/N -2 \delta}).
  $$
\end{lem}
To prove this lemma we recall a lemma of Athreya, Fairchild and Masur \cite{athreya2022counting} bounding $$N(\omega, R) := \Lambda_{\omega} \cap B(0,R) $$ with respect to the length of the smallest saddle connection.

\begin{lem}[Lemma 5.4 of \cite{athreya2022counting}]
  \label{lem_ContingShortest}
  For any $L_0 >0$ and $\delta>0$, there exist $C(\delta,L_0)$ such that for any $L<L_0$ and any surface $(X,\omega)$ in the stratum we have
  $$
  N(\omega,L) \leq C \left( \frac{L}{l(\omega)} \right)^{1+\delta}.
  $$
\end{lem}

We should also control $N(\omega, R)$ with respect to the length of the shortest saddle connection non-parallel to $\gamma$.

\begin{lem}
  \label{lem_ContingSecShortest}
  For $\omega \in \HH$, if $\gamma$ is its shortest saddle connection and $\gamma'$ is the shortest saddle connection non-parallel to $\gamma$, if $\gamma$ does not bound a cylinder or if $\gamma$ bound a cylinder of width at least $\epsilon_0$, then $$
  N(\omega,\epsilon_0)^2 = O(l(\gamma')^{-2N})
  $$
\end{lem}

\begin{proof}
  The proof has been done in Section \ref{sec2} in the paragraphs \newline \textbf{"The shortest non-parallel saddle connection is longer that a power of the shortest and the shortest is not a cylinder curve"} \newline and \newline \textbf{"The shortest non-parallel saddle connection is longer that a power of the shortest and the shortest is a cylinder curve and the height of the cylinder is at least $\epsilon_0$"}.
\end{proof}

We can now adapt the proof of Lemma 5.5 of \cite{athreya2022counting} to obtain our equivalent Lemma \ref{lem_IntThinpart}.

\begin{proof}[Proof of Lemma \ref{lem_IntThinpart}]
  Following the proof of Athreya, Fairchild and Masur, we choose a real $\sigma \in ]0,1[$, and we define three families of set exhausting $\HH_{\epsilon}$.

  We will use the notation of the proof of theorem \ref{thm_SV_epsilon} that is $\gamma$ is the shortest saddle connection of a translation surface $\omega$ and $\epsilon$ is the shortest saddle connection non-parallel with $\gamma$.

   The first family is $$
  F(j) =: \{ \omega \in \HH,  \sigma^{j+1} \leq |\gamma| \leq \sigma^j, |\epsilon| \leq \sigma^{j/2N} \}.
  $$
  In this case Lemma \ref{lem_IntThinpart} say that $$
  \mu(F(j)) = O(\sigma^{2j+2j/2N})
  $$
  and Lemma \ref{lem_ContingShortest} that for $\omega \in F(j)$$$
  \psi(\omega)= O(\sigma^{-j(2+2\delta)}).
  $$
  This gives $$
  \int_{F(j)} \psi d \mu = O(\sigma^{j(1/N - 2\delta)}).
  $$

  The second family is
  \begin{align*}
    G(j) := \{ \omega \in \HH, \sigma^{j+1} \leq |\gamma| \leq \sigma^j, |\epsilon| \leq \sigma^{j/2N}, \gamma \textrm{ is not on the boundary of a cylinder or} \\
   \textrm{  is on the boundary of a cylinder of height } \geq \epsilon_0 \}.
  \end{align*}
  In this case Lemma \ref{lem_Dozier} gives that $$
  \mu(G(j))= O(\sigma^{2j})
  $$
  and Lemma \ref{lem_ContingSecShortest} indicates that for $\omega \in G(j)$
  $$
  \psi(\omega)= O((\sigma^{j/2N})^{-2N})=O(\sigma^{-j}).
  $$
  And so $$
  \int_{G(j)} \psi d \mu = O(\sigma^j).
  $$
  Finally the third family is
  $$
  H(j) := \{ \omega \in \HH, \sigma^{j+1} \leq |\gamma| \leq \sigma^j, \gamma \textrm{ is on the boundary of a cylinder of height } \leq \epsilon_0 \}.
  $$
  In this case Lemma \ref{lem_ContingShortest} gives that for $\omega \in H(j)$
  $$
  \psi(\omega) = O(\sigma^{-j(2+ 2 \delta)}).
  $$
  Moreover the measure of this set can be commputed as the measure of the set $\Omega_3(k)$ is the last case of the proof of Theorem \ref{thm_SV_epsilon}. Notice that in the computation of the volume of $\Omega_3(k)$ the length of the shortest non-parallel saddle connection does not interfere.
  So $$
  \mu(H(j))= O(\sigma^{3j})
  $$
  and
  $$
  \int_{H(j)} \psi d \mu = O(\sigma^{j(1-2\delta}).
  $$
  Finally choose $j_0$ so that $\sigma^{j_0+1} \leq \epsilon_1 \leq \sigma^{j_0}$ we have
  $$
  \int_{\HH_{\epsilon_1}} \psi d \mu = O(\sum_{j \geq j_0}
   \sigma^j + \sigma^{j(1-2\delta)} + \sigma^{j(1/N - 2\delta)})= O(\sigma^{j_0(1/N - 2\delta)})=O(\epsilon_1^{1/N - 2\delta}).
  $$
\end{proof}

The third one (Lemma 5.9 of \cite{athreya2022counting}) bounds the measure of a set described by four inequalities.

More precisely, given $L$, $\hat{\epsilon}$, $\epsilon'$ and $L' \in \{1/2, 1\}$ define $\Omega(\hat{\epsilon},\epsilon', L, L')$ to be the set of surfaces $\omega$ such that
$\omega$ is $\hat{\epsilon}$-thick (that is has no saddle connection of length less than $\hat{\epsilon}$), and there are $(z, w) \in \left(\Lambda_\omega \cap B(0, L)\right)^2$, where at least one of the following holds:
\begin{itemize}
    \item $1-\epsilon'\leq \frac{|\Imm(w)|}{|\Imm(z)|} \leq 1+ \epsilon'$
    \item $(1-\epsilon')A\leq |z \wedge w| \leq (1+ \epsilon')A$
    \item $|\Imm(z)-L'|<\epsilon'$
    \item $(1-\epsilon') \Imm(z)\leq |\Ree(z)| \leq (1+\epsilon')\Imm(z)$
\end{itemize}
\begin{lem}
  \label{lem_MeasureOmega}
 There exists $C$ and $D>0$ so that for all $\hat{\epsilon},\epsilon'$
 $$
\mu(\Omega(\hat{\epsilon}, \epsilon', L, L')) \leq C \frac{\epsilon'}{\hat{\epsilon}^D}
 $$
\end{lem}
\begin{proof}
  As the $\hat{\epsilon}$-thick part of the moduli space is compact and the affine manifold $\MM$ is closed, their intersection is also compact.
  Each point has a neighborhood in which there is a triangulation that is a Delaunay triangulation for every point in this neighborhood. We fix a finite cover of this family of neighborhood.
  Then taking one of these neighborhood, we will work with a fixed Delaunay triangulation.

Let's take a maximal family of $\MM$-independent saddle connections which are edges of the Delaunay triangulation $t_1, t_2,\cdots, t_k$ so that the $\mu$ measure is the Lebesgue measure under their period coordinates. We decompose the period coordinates as $t_j=x_j + i y_j$.

  Then we complete this family with $t_{k+1},\cdots, t_n$ to have a basis of the $H_1(X,\Sigma,\mathbb{C})$.

  The length of any $t_i$ is bounded below by $\hat{\epsilon}$ and by the Paragraph 4 of \cite{masur_smilie} also bound above by $1/\hat{\epsilon}$.

  Then if two saddle connections $z, w$ have length less than $L$ and satisfies the first inequality, we can decompose $z=\sum_{i=1}^n z_i t_i$ and $w=\sum_{i=1}^n w_i t_i$ in the homology, with integer coefficient.

  The coefficients are $O(\frac{1}{\hat{\epsilon}})$ since the Delaunay triangulation is efficient.

  Since every $t_i$, for $k<i\leq n$ are linear combinations of the $t_j$ for $1 \leq j \leq k$, we can write $z=\sum_{i=1}^k z_i' t_i$ and $w=\sum_{i=1}^k w_i' t_i$ with coefficients which are still $O(\frac{1}{\hat{\epsilon}})$. We fix a pair of tuple of coefficients $(z_i',w_i')$ with $z_i' \ne w_i'$.

  We are interested in computing the Lebesgue volume of vectors $s_i$ of lengths less than $L$ such that $$
  \left \vert \frac{\sum_{i=1}^k z_i' y_i}{\sum_{i=1}^k w_i' y_i} - 1 \right \vert \leq \epsilon'
$$

  This volume is smaller than the volume of the set of vectors of lengths less than $L$ such that $$
  \left \vert \sum_{i=1}^k (z_i'-w_i') y_i \right \vert \leq \epsilon' O(\frac{1}{\hat{\epsilon}^2})
$$

  This is the intersection of the sphere of radius $L$ centered at $0$ and the $O(\frac{\epsilon'}{\hat{\epsilon}^2})$ neighborhood of a hyperplane in a $k$ dimensional coordinate vector space, whose volume is $O(\frac{\epsilon'}{\hat{\epsilon}^2})$.
  The number of choices of coefficient is $O(\frac{1}{\hat{\epsilon}^k})$. By summing the estimates for each choice, we get the bound of the first point.

  The other points lead to same the kind of inequalities and are demonstrated  similarly.

\end{proof}

At this point we have shown that
  given $\mu$ an ergodic $\SLR$-invariant measure, there is a $c_{\mu}(A)$ such that for $\mu$-almost every $\omega \in \HH$
  $$
  \underset{R \to \infty}{\lim} \frac{N_{A}(\omega, R)}{R^2} = c_{\mu}(A).
  $$

\section{Effectivization of the counting}
\label{sec5}

  In order to get the power saving error term of Theorem \ref{thm2}, we mainly need to control two quantities.
The first one is the difference between the circle average and the integral of the Siegel-Veech transform of $h_A$. To do this we will use an effective version of Nevo ergodic theorem \ref{thm_Nevo}.
The second quantity is the sum of the four error terms $e_t^i$.

   \subsection{Effective Nevo ergodic theorem}

  We define now the class of function, we should work with.

  \begin{dfnt}
    A function $f \in L^2(\HH,\mu)$ is called K-$smooth$ of degree one if
   $$
   \pi_{\HH}(\theta)f := \underset{\theta \to 0}{\lim}\frac{1}{\theta}(r_{\theta}^* f -f)
   $$
     %$$
    %\pi_x(\Omega)f := \underset{\phi \to 0}{\lim} \frac{1}{\phi} (\pi_X exp(\phi \Omega)f-f)
    %$$
    exists with respect to the $L^2(\HH,\mu)$-norm.

    We define the \emph{Sobolev norm} by $$
    S_K(f)^2 = \|f \|^2_2 +\| \pi_{\HH}(\theta)f \|_2^2
    $$

    The $K$\emph{-Sobolev space} is then define as $$
    \Sigma_{K}(\HH)=\{f \in L^2(\HH,\mu), S_{K}(f)< +  \infty\}.
    $$
    For a function $g$ from $\mathbb{C}^2$ to $\mathbb{R}$ we define $$
    \partial_{\theta} g := \underset{\theta \to 0}{\lim} \frac{r_{\theta}^{*}g-g}{\theta}.
    $$
    Note that $(\partial_{\theta} g)^{SV}=\pi_{\HH}(\theta) \hat{g}$.
  \end{dfnt}

  We can bound the Siegel-Veech transform of function with compact support and its $S_K$-norm of  with data from the original function.

  \begin{lem}
    \label{lem_SKnorme}
    Given $\alpha_1>2$, let $g: \mathbb{C}^2 \to \mathbb{R}$ be function with compact support and let $\epsilon>0$.
     We have that
     $$
     |\hat{g}(\omega)| \leq C_{supp(g)} l(\omega)^{- \alpha_1} \|g\|_{\infty}
     $$
     and if $g$ is differentiable,
     $$
    S_K \left( \hat{g}\chi_{l(\cdot) \geq \epsilon}-\int_{\HH} \hat{g} \chi_{l(\cdot) \geq \epsilon} d\mu \right)^2 = O( \frac{\|g\|_{\infty}^2 +\|\partial_{\theta}g\|_{\infty}^2}{\epsilon^{2 \alpha_1}})
    $$
    where the implied constant depend on the support.

  \end{lem}
  \begin{proof}
    For the first point, taking any surface $\omega \in \HH$, if a pair of saddle connection are in $supp(g)$, then they both belong to a compact set $B \subset \mathbb{R}^2$.
      Hence using Lemma \ref{lem_ContingShortest}, we got
      $$
      | \hat{g}(\omega) | \leq \|g\|_{\infty} (\# \Lambda(\omega) \cap B)^2 \leq C_{supp(g)} l(\omega)^{-\alpha_1}\|g\|_{\infty}
      $$

      For the second point

    \begin{align*}
      \|\hat{g}\chi_{l(\cdot) \geq \epsilon}-\int_{\HH} \hat{g} \chi_{l(\cdot) \geq \epsilon} d\mu \|_2^2 & \leq \int_{l(\cdot) \geq \epsilon} \hat{g}^2 d\mu + \left( \int_{\HH} \hat{g} \chi_{l(\cdot) \geq \epsilon} d\mu \right)^2\\
       & \leq O(\int_{l(\cdot) \geq \epsilon} l(\omega)^{-2\alpha_1} \|g\|_{\infty}^2 d\mu) + O( \left( \int_{l(\cdot) \geq \epsilon} l(\omega)^{-\alpha_1} \|g\|_{\infty}  d\mu\right)^2)\\
       & = O (\|g\|_{\infty}^2  \epsilon^{- 2 \alpha_1})
    \end{align*}
Where the second inequality comes from the first point of this lemma.

Then note that  $\pi_{\HH}(\theta)\chi_{l(\cdot) \geq \epsilon}=0$ and $\pi_{\HH}(\theta)\hat{g}= (\partial_{\theta}g)^{SV}$ and so $$\pi_{\HH}(\theta)\hat{g}\chi_{l(\cdot) \geq \epsilon}=(\partial_{\theta}g)^{SV}\chi_{l(\cdot) \geq \epsilon}.$$

    So, in the same flavor as the previous estimate
    \begin{align*}
  \|\pi_{\HH}(\theta)\hat{g}\chi_{l(\cdot) \geq \epsilon} \|_2^2 & = \int_{l(\cdot) \geq \epsilon} \left( \partial_{\theta}g^{SV} \right)^2 d\mu \\
   & \leq \int_{l(\cdot) \geq \epsilon}  l(\omega)^{-2 \alpha_1} \|\partial_{\theta}g\|_{\infty}^2 d\mu \\
   &= O(\epsilon^{-2 \alpha_1} \|\partial_{\theta}g\|_{\infty}^2)
 \end{align*}
This end the demonstration of this lemma.

  \end{proof}

The following statement is an effective version of \ref{thm_Nevo} which is included in Theorem 3.5 of \cite{nevo2020effective}.

   \begin{thm}
     \label{thm_effNevo}
     There is a $\lambda' >0$ and a $C>0$ such that for all $t>1$ and any $f \in \Sigma_K(\HH)$, we have $$
     \| A_t f - \int_{\HH} f d \mu \|^2_2 \leq C e^{-2 \lambda' t} S_K(f)^2.
     $$
     Furthermore, if $(t_n)\subset{\mathbb{R}_+}$ and $\eta_1$ satisfy
     \begin{equation}
       \label{eq_ConditionTemp}
       \sum_{n \in \mathbb{N}} e^{- \lambda' \eta_1 t_n} < \infty
    \end{equation}
     then for almost all $\omega \in \HH$ there is $n_0=n_0(\omega)$ such that for all $n \geq n_0$
     $$
     | A_{t_n} f(\omega) - \int_{\HH} f d \mu | \leq C e^{-2(\eta-\frac{\eta_1}{2}) \lambda' t_n} S_K(f)^2
     $$
     with $\eta=\frac{1}{\lambda'+1}$
   \end{thm}

   From now on, we fix a sequence $(t_n)$ and an $\eta_1$ satisfying equation \ref{eq_ConditionTemp}. We will call $\lambda= \frac{\eta-\eta_1 /2}{\lambda'+1}$.

  \subsection{Circle average convergence}

  The aim of this subsection is to estimate $$
  \left\vert A_t \hat{h}_A(\omega) - \int_{\HH}\hat{h}_A d \mu \right\vert
  $$

  To begin, we will need a bounding lemma (Lemma \ref{lem_CircleAverageThin} below) which is based on the following theorem.

  \begin{thm}[See \cite{eskin_masur_2001} Thm 5.2 and Lem 5.5 and \cite{VeechQuad} Cor 2.8]
    \label{thm_IntSys}
      For any $\omega \in \HH$, and for any $1 \leq \alpha < 2$,$$
      \underset{t>0}{\sup} A_t(l^{-\alpha})(\omega) < \infty.
      $$
      The bound can be taken uniform as $\omega$ ranges over compact sets in $\HH$. Moreover,
      $$l(\cdot)^{-\alpha} \in L^1(\HH,\mu).$$

  \end{thm}

  Using this theorem we can bound the circle average of a Siegel-Veech transform of a function with compact support on the thin part.
  \begin{lem}
    \label{lem_CircleAverageThin}
    There is an $\alpha_2>0$ such that for any $g: \mathbb{C}^2 \to \mathbb{R}$ be a positive function with compact support, any $\epsilon>0$, any $\omega \in \HH$ there is a constant $K(\omega,g)$ with
    $$
    |A_{t_n}(\hat{g} \chi_{l(\cdot)\leq \epsilon})| \leq \epsilon^{\alpha_2} K(\omega,g).
    $$
    Moreover if $h \leq g$ then $K(\omega,h) \leq K(\omega,g)$.
  \end{lem}
  \begin{proof}
    Denote by $\chi_{\epsilon} := \chi_{l(\cdot)\leq \epsilon}$ and observe that, by taking $\alpha_2$ such that $\alpha_2 \kappa < 2(1+\kappa)$, one has
    \begin{align*}
      |A_t(\hat{g}\chi_{\epsilon})(\omega)| & = \frac{1}{2\pi} \int_{0}^{2\pi} \hat{g}(g_t r_{\theta}\omega) \chi_{\epsilon}(g_t r_{\theta}\omega) d \theta \\
      & \leq \frac{\epsilon^{\alpha_2}}{2\pi} \int_{0}^{2\pi} \hat{g}(g_t r_{\theta}\omega) l^{-  \alpha_2}(g_t r_{\theta}\omega) \chi_{\epsilon}(g_t r_{\theta}\omega) d \theta \\
      & \leq \frac{\epsilon^{\alpha_2}}{2\pi} \left( \int_{0}^{2\pi}  l^{ \frac{\alpha_2 \kappa}{^{1+\kappa}}}(g_t r_{\theta}\omega) d \theta \right)^{\frac{\kappa+1}{\kappa}}  \left( \int_{0}^{2\pi}
      \hat{g}(g_t r_{\theta}\omega)^{1+\kappa} \chi_{\epsilon}(g_t r_{\theta}\omega)^{1+\kappa} d \theta \right)^{\frac{1}{1+\kappa}}\\
    \end{align*}
  where the second inequality is the Hölder inequality. Then by Theorem \ref{thm_IntSys} and the choice of $\alpha_2$ the first integral of the last line is finite. The second integral, with Theorem \ref{cor_LimAver}, converges to $$
  \int_{\HH} \hat{g}^{1+\kappa} \chi_{\epsilon}^{1+\kappa} d\mu.
  $$
  So the second integral is bound for all $t$ by a quantity called $K'(\omega,g)$.
  This completes the argument.

  The last affirmation is obvious looking at what defined the constants.
  \end{proof}

   As $\hat{h}_A$ does not satisfy the hypothesis of Theorem \ref{thm_effNevo}, we will approach $\hat{h}_A$ by a family of function depending on two parameters. The first one will smooth $h_A$ and the second will restrict its Siegel-Veech transform to the thick part of the moduli space.

   \begin{dfnt}
     Let $\eta: \mathbb{C}^2 \to [0,1]$ be a smooth function with support in $B(0,2)$ and equal to $1$ on $B(0,1)$ and such that $$
     \int_{\mathbb{C}^2} \eta d Leb = 1.
     $$
      Let $$
     \eta_s(y)= \frac{1}{s^4}\eta(\frac{y}{s})
     $$
     and
     $$
      R_{A,s}(\mathcal{T}) := \{ x \in \mathbb{C}^2, dist(x,R_A(\mathcal{T})) \leq s \}
     $$
     We consider
     $$
     g_{A,s}(x)= (\chi_{R_{A,s}(\mathcal{T})} * \eta_s)(x) = \int_{\mathbb{C}^2} \eta_s(x-y)\chi_{R_{A,s}(\mathcal{T})}(y) dLeb(y)
     $$

  $g_{A,s}$ is a function from $\mathbb{C}^2$ to $[0,1]$ which is differentiable, with a bounded differential and compact support. An easy consequence is that the $K$-derivative, $\partial_{\theta}g_{A,s}$ is also bound and is $O(\frac{1}{s})$.
\end{dfnt}

The supports of the function $g^t_A$ are increasing meaning that if $s<t<1$ then $$
supp(g_{s,A}) \subset supp(g_{t,A})
  $$
  In particular $supp(g_{t,A}) \subset supp(g_{1,A})$ for every $t \leq 1$.
  We bound every function $g_{t,A}$ by the characteristic function of $supp(g_{1,A})$.

  So using the first part of Lemma \ref{lem_SKnorme} we have that $$
|\hat{g}_{A,s}(\omega)| \leq C l(\omega)^{- \alpha_1}
  $$
with no dependency on $s$ for the constant.

This family of function also satisfy the following estimates.

\begin{lem}
  \label{lem_diffHetg}
  There are $\kappa_1,\kappa_2,\kappa_3>0$ such that, for almost every $\omega \in \HH$ there is a $n_0$ such that for $n \geq n_0(\omega)$ we have:
  \begin{itemize}
    \item $\left\vert A_{t_n} \hat{h}_A(\omega) - A_{t_n} \hat{g}_{A,s}(\omega) \right\vert \leq O_{\omega}(s^{\kappa_1})+  O(s^{-\kappa_2})e^{-2\lambda t_n}$
    \item $\left\vert \int_{\HH}\hat{g}_{A,s} d \mu  - \int_{\HH}\hat{h}_{A} d \mu \right\vert \leq O(s^{\kappa_3})$
  \end{itemize}
\end{lem}
\begin{proof}

  Let's take
  $$
   F_{A,s}(\mathcal{T}) := \{ x \in \mathbb{C}^2, dist(x,\partial R_A(\mathcal{T})) \leq s \}
  $$
  We consider
  $$
  \psi_s(x)= (\chi_{F_{A,s}(\mathcal{T})} * \eta_s)(x)
  $$
  We have that $0 \leq g_{A,s}-h_A \leq \psi_s$, $\psi_s \leq 1$, the support of $\psi_s$ is compact, $\psi_s$ is differentiable with bounded differential. So,calling $\chi_{\epsilon}$ the characteristic function of the $\epsilon$-thin part, we can say using Theorem \ref{thm_effNevo} that for almost every $\omega \in \HH$ and $n\geq n_0(\omega)$
  $$
  \left\vert A_{t_n} \hat{h}_A(\omega) - A_{t_n} \hat{g}_{A,s}(\omega) \right\vert \leq  \left\vert A_{t_n} \hat{\psi}_s(\omega)\chi_{\epsilon} \right\vert + \left\vert A_{t_n} \hat{\psi}_s(\omega)(1-\chi_{\epsilon}) \right\vert
  $$
  As $\psi_s \leq \psi_1$, by Lemma \ref{lem_CircleAverageThin} we already have that
  $$
  \left\vert A_{t_n} \hat{\psi}_s(\omega)\chi_{\epsilon} \right\vert = O(\epsilon^{\alpha_2})K(\omega,\psi_s) \leq O(\epsilon^{\alpha_2})K(\omega,\psi_1)
  $$
  And using Theorem \ref{thm_effNevo},
  $$
  \left\vert A_{t_n} \hat{\psi}_s(\omega)(1-\chi_{\epsilon}) \right\vert
  \leq \left\vert  \int_{l(\cdot) \geq \epsilon} \hat{\psi}_s d \mu \right\vert
  + C S_K \left( \hat{\psi}_s (1-\chi_{\epsilon}) -\int_{\HH} \hat{\psi}_s (1-\chi_{\epsilon}) d \mu \right)^2 e^{-2 \lambda t_n}.
  $$
    Using the second part of Lemma \ref{lem_SKnorme}, we have easily that $$S_K \left( \hat{\psi}_s (1-\chi_{\epsilon}) -\int_{\HH} \hat{\psi}_s (1-\chi_{\epsilon}) d \mu \right)^2= O(\frac{1}{\epsilon^{2 \alpha_1}s^2}).$$

  Moreover $$
  \left\vert \int_{\HH}\hat{g}_{A,s} d \mu  - \int_{\HH}\hat{h}_{A} d \mu \right\vert \leq \int_{\HH} \hat{\psi}_s d \mu = \int_{l \leq \epsilon}\hat{\psi}_s d \mu + \int_{l \geq \epsilon}\hat{\psi}_s d \mu.
  $$

  For the first integral by Lemma \ref{lem_IntThinpart}, if $N$ is the maximal number of edge in a Delaunay triangulation and $\delta$ satisfies $\delta < \frac{1}{2N}$ then $$
  \int_{l \leq \epsilon}\hat{\psi}_s d \mu = O(\epsilon^{ 1/N -2 \delta}).
  $$
  Then by Lemma \ref{lem_SKnorme}, for each $\omega$ in the $\epsilon$ thick part, $|\hat{\psi}_s(\omega)|=O_{supp(\psi_s)}(\frac{1}{\epsilon^{\alpha_1}})=O_{supp(\psi_1)}(\frac{1}{\epsilon^{\alpha_1}})$. As the support of $\psi_s$ is included in $\Omega(\epsilon_1,s,L,L')$ for a $L$ big enough,
   Lemma \ref{lem_MeasureOmega} say that the measure of the support of $\hat{\psi}_s$ is $O(\frac{s}{\epsilon_1^{D}})$.
   Putting everything together we have that
  $$
  \int_{\HH}\hat{\psi}_s d \mu \leq O(\epsilon^{ 1/N -2 \delta})+O(\frac{s}{\epsilon^{D+\alpha_1}})
  $$
  Taking $\epsilon=s^{\frac{1}{D+\alpha_1+1/N-2\delta}}$ we got that $$
  \int_{\HH}\hat{\psi}_s d \mu \leq O(s^{\frac{1/N-2\delta}{D+\alpha_1+1/N-2\delta}}).
  $$
  This concludes the second point.

  For the first point we have
  $$
\left\vert A_{t_n} \hat{h}_A(\omega) - A_{t_n} \hat{g}_{A,s}(\omega) \right\vert \leq O(\epsilon^{\alpha_2})K(\omega)+ O(\frac{s}{\epsilon^{D+\alpha_1}}) +O(\frac{1}{\epsilon^{2 \alpha_1}s^2}) e^{-2\lambda t_n}
  $$
  Taking here $\epsilon=s^{\frac{1}{D+2\alpha_1+\alpha_2}}$ we have
  $$
\left\vert A_{t_n} \hat{h}_A(\omega) - A_{t_n} \hat{g}_{A,s}(\omega) \right\vert \leq O_{\omega}(s^{\frac{\alpha_2}{D+2\alpha_1+\alpha_2}})
+O(\frac{1}{s^{\frac{\alpha_1}{D+2\alpha_1+\alpha_2}+2}})
e^{-2\lambda t_n}
  $$
\end{proof}

We now fix an $\epsilon>0$ and call $\chi_{\epsilon}$ the characteristic function of $l^{-1}(]0,\epsilon[)$ on $\HH$. We have that,

\begin{align}
  \label{estimeff}
  \left\vert A_t \hat{h}_A(\omega) - \int_{\HH}\hat{h}_A d \mu \right\vert & \leq \left\vert A_t \hat{h}_A(\omega) - A_t \hat{g}_{A,s}(\omega) \right\vert + \left\vert A_t \hat{g}_{A,s}\chi_{\epsilon}(\omega) \right\vert \nonumber \\
   & + \left\vert A_t \hat{g}_{A,s}(1-\chi_{\epsilon})(\omega) - \int_{\HH}\hat{g}_{A,s}(1-\chi_{\epsilon}) d \mu \right\vert \\
   & + \left\vert \int_{\HH}\hat{g}_{A,s}\chi_{\epsilon} d \mu \right\vert + \left\vert \int_{\HH}\hat{g}_{A,s} d \mu  - \int_{\HH}\hat{h}_{A} d \mu \right\vert \nonumber
\end{align}
In the previous inequality we approximated $h_A$ by $g_{A,s}$ and then treated separately the thin and thick part of the moduli space. We can bound these terms with the following lemma.

\begin{lem}
  There are $\alpha_3,C>0$ such that
      \begin{itemize}
        \item $S_K( \hat{g}_{A,s}(1-\chi_{\epsilon})-\int_{\HH}\hat{g}_{A,s}(1-\chi_{\epsilon}) d \mu)^2 \leq \epsilon^{- \alpha_1}O(\frac{ 1}{s^2})$
        \item $ \int_{\HH} \hat{g}_{A,s}\chi_{\epsilon} d \mu \leq C_3 \epsilon^{\alpha_3}$
        \item $|A_t(\hat{g}_{A,s}\chi_{\epsilon})(\omega)| \leq \epsilon^{\alpha_2} K(\omega)$
      \end{itemize}
    \end{lem}

    \begin{proof}
      The first point is just Lemma \ref{lem_SKnorme}.

      For the second point, choosing $\alpha_3$ such that $\alpha_3 + \alpha_1 < 2$, with Theorem \ref{thm_IntSys} we have
      \begin{align*}
          \int_{\HH} \hat{g}_{A,s}\chi_{\epsilon} d \mu & \leq \int_{l(\omega \leq \epsilon)} \hat{g}_{A,s} d \mu \\
          & \leq \int_{l(\omega \leq \epsilon)} C l(\omega)^{- \alpha_1} d \mu \\
          & \leq \int_{l(\omega \leq \epsilon)} C \frac{\epsilon^{\alpha_3}}{l(\omega)^{ \alpha_3}}l(\omega)^{-  \alpha_1} d \mu \\
          & \leq \epsilon^{\alpha_3} C \| l^{-\alpha_3 - \alpha_1} \|_1
      \end{align*}
      The third point is an application of Lemma \ref{lem_CircleAverageThin}, in which we bound $g_{A,s}$ by  $g_{A,1}$ to get rid of the dependency on $s$.
    \end{proof}

Going back to inequality \ref{estimeff}, we have that for almost every $\omega \in \HH$ and $n \geq n_0(\omega)$

\begin{align*}
  \left\vert A_{t_n} \hat{h}_A(\omega) - \int_{\HH}\hat{h}_A d \mu \right\vert & \leq \left\vert A_{t_n} \hat{h}_A(\omega) - A_{t_n} \hat{g}_{A,s}(\omega) \right\vert + \left\vert A_{t_n} \hat{g}_{A,s}\chi_{\epsilon}(\omega) \right\vert \\
  & + \left\vert A_{t_n} \hat{g}_{A,s}(1-\chi_{\epsilon})(\omega) - \int_{\HH}\hat{g}_{A,s}(1-\chi_{\epsilon}) d \mu \right\vert \\
   & + \left\vert \int_{\HH}\hat{g}_{A,s}\chi_{\epsilon} d \mu \right\vert + \left\vert \int_{\HH}\hat{g}_{A,s} d \mu  - \int_{\HH}\hat{h}_{A} d \mu \right\vert \\
   & \leq O_{\omega}(s^{\kappa_1})+ O(s^{-\kappa_2})e^{-2\lambda t_n} + \epsilon^{\alpha_2} K(\omega) \\
   & + C e^{-2 \lambda t_n }S_K( \hat{g}_{A,s}(1-\chi_{\epsilon})-\int_{\HH}\hat{g}_{A,s}(1-\chi_{\epsilon}) d \mu)^2\\
   & + C_3 \epsilon^{\alpha_3} + O(s^{\kappa_3})\\
   & \leq  e^{-2 \lambda {t_n} }O(s^{-\kappa_2}+\frac{1}{s^2\epsilon_1})+O_{\omega}(s^{\kappa_1}+s^{\kappa_3}+\epsilon^{\alpha2}+\epsilon^{\alpha_1})
\end{align*}

Putting $\epsilon=e^{-f_1 t_n}$ and $s=e^{-f_2 t_n}$ we have
$$
  \left\vert A_{t_n} \hat{h}_A(\omega) - \int_{\HH}\hat{h}_A \right\vert \leq O_{\omega}(e^{-t_n(2\lambda-\kappa_2 f_2)} + e^{-t_n(2\lambda-2 f_2-\alpha_1 f_1)} + e^{-f_2 \kappa_1 t_n} + e^{-f_2 \kappa_3 t_n} + e^{ - f_1 \alpha_2 t_n} + e^{ - f_1 \alpha_3 t_n})
$$
%So to maximize the exponent we have to take $f_1=\frac{2\lambda}{4 \alpha_1 + \alpha_2}$ and $f_2=\frac{2 \lambda}{4+ \kappa}$, calling $f= \max(\frac{2\alpha_2\lambda}{4 \alpha_1 + \alpha_2},\frac{2 \lambda \kappa}{4+\kappa})$
Choosing wisely $f_1$ and $f_2$ we have for a $f>0$
$$
  \left\vert A_{t_n} \hat{h}_A(\omega) - \int_{\HH}\hat{h}_A \right\vert \leq O_{\omega}(e^{-t_n f}).
$$

\subsection{Error terms}

Equation (4.10) of \cite{athreya2022counting} indicates that $$
\left\vert N_A^*(\omega,e^t) - \pi e^{2t}(A_t \hat{h}_A)(\omega) \right\vert \leq |m_t(\omega)| + \sum_{i=1}^4 |e^i_t(\omega)|
$$
where the different terms are explained above.

Equation (5.4) of \cite{athreya2022counting} gives that $|m_t(\omega)|= O(e^{-2t})$. We will concentrate on the four other terms.

For the other error terms we choose $\hat{\epsilon}>0$ to work separately on the thin and thick parts of the moduli space.

If $\omega$ is on the $\hat{\epsilon}$ thin part of the moduli space, according to Corollary 5.8 of \cite{athreya2022counting}, we have that
$$
|E^k_t \cap \Lambda_{\omega}^2|=O(\hat{\epsilon}^{1/N-2\delta}e^{2t}).
$$

We should now pick $\epsilon'$ satisfying every inequality in the proof of theorem 5.1 in \cite{athreya2022counting} which are on the beginning of each paragraph "Error term $E^i_t$".

Taking larger bounds than these inequalities, we can find a $K$ such that $\epsilon'$ satisfies all the conditions if
$$
  K e^{-2t} \leq \epsilon'.
$$

To have estimates on the error terms on the thick part, it is needed to bound $|A_t h|$ where $h$ is the characteristic function of $\Omega(\frac{\hat{\epsilon}}{\sqrt{8}\pi}, \epsilon',L,L')$. As $h$ is not $K$-smooth we will also approach it by a family of function defined below.

To do that let's call $\Theta(\hat{\epsilon},\epsilon',L,L')$ the set of pair $(w,z)$ in $\mathbb{C}^2$ such that
$$\hat{\epsilon} \leq \min(|w|,|z|) \leq \max(|w|,|z|) \leq L$$
 and which verified at least one of these equations
  \begin{itemize}[label = $\bullet$]
      \item $1-\epsilon'\leq \frac{|\Imm(w)|}{|\Imm(z)|} \leq 1+ \epsilon'$
      \item $(1-\epsilon')A\leq |z \wedge w| \leq (1+ \epsilon')A$
      \item $|\Imm(z)-L'|<\epsilon'$
      \item $(1-\epsilon') \Imm(z)\leq |\Ree(z)| \leq (1+\epsilon')\Imm(z)$
  \end{itemize}

We consider $\eta_s$ the family of bump functions defined in the previous part and call $\xi_{s} = \eta_{s} * \chi_{ \{(z,w), dist((z,w),\Theta(\frac{\hat{\epsilon}}{\sqrt{8}\pi}, \epsilon',L,L')) \leq s\}}$.

$\xi_s$ have the following properties
\begin{enumerate}
  \item $\xi_s(z,w) \in [0,1]$
  \item $\xi_s\geq \chi_{\Theta(\frac{\hat{\epsilon}}{\sqrt{8}\pi},\epsilon',L,L')}$
  \item $supp(\xi_s) \subset \Theta(\frac{\hat{\epsilon}}{\sqrt{8}\pi}-s,\epsilon'+\frac{s}{\hat{\epsilon}},L+s,L') \subset \Theta(\frac{\hat{\epsilon}}{2\sqrt{8}\pi},\epsilon'+\frac{s}{\hat{\epsilon}},2L,L')$ for $s$ small enough
  \item $\xi_s$ is differentiable with its norm bound by $O(1/s)$.
\end{enumerate}

Then $h$ is bound by the Siegel-Veech transform of these functions time the characteristic function of the $\frac{\hat{\epsilon}}{\sqrt{8}\pi}$-thick part of the moduli space
$$\chi_{l(\cdot \geq \frac{\hat{\epsilon}}{\sqrt{8}\pi})} \hat{\xi}_s := \Xi_{s,\hat{\epsilon}}.$$

$\Xi_{s,\hat{\epsilon}}$ has its support on $\Omega(\frac{\hat{\epsilon}}{\sqrt{8}\pi}, \epsilon' + \frac{\sqrt{8}\pi s}{\hat{\epsilon}}, 2L,L')$ and is $K$-smooth.

Using the first part of Lemma \ref{lem_SKnorme} we have that for all $\omega \in \HH$ $$|\Xi_{s,\hat{\epsilon}}(\omega)| \leq O(\hat{\epsilon}^{- \alpha_1}).$$

Using the second part of Lemma \ref{lem_SKnorme} we have that $S_K(\Xi_{s,\hat{\epsilon}} - \int_{\HH}\Xi_{s,\hat{\epsilon}} d \mu)^2 = O(\frac{1}{\hat{\epsilon}^{2 \alpha_1}s^2})$

Moreover using Lemma \ref{lem_MeasureOmega} we got $$
\int_{\HH} \Xi_{s,\hat{\epsilon}} d\mu  \leq \|\Xi_{s,\hat{\epsilon}}\|_{\infty}\mu(\Omega(\frac{\hat{\epsilon}}{\sqrt{8}\pi}, \epsilon' + \frac{s}{\hat{\epsilon}}, 2L,L')) = O(\hat{\epsilon}^{-\alpha_1})
 O(\frac{\epsilon'+\frac{s}{\hat{\epsilon}}}{\hat{\epsilon}^D}).
$$

We have that for $\mu$-almost every $\omega \in \HH$ that for $n \geq n_0(\omega)$
$$|A_{t_n} \hat{h}| \leq |A_{t_n} \Xi_{s,\hat{\epsilon}}| \leq \int_{\HH} \Xi_{s,\hat{\epsilon}} d \mu
+ C S_K(\Xi_{s,\hat{\epsilon}} - \int_{\HH}\Xi_{s,\hat{\epsilon}} d \mu)^2 e^{-\lambda t_n} =
 O(\frac{\epsilon'+\frac{s}{\hat{\epsilon}}}{\hat{\epsilon}^{D+ \alpha_1}} + \frac{e^{-\lambda t_n}}{ \hat{\epsilon}^{2 \alpha_1} s^2}). $$

 Taking $s=\epsilon' \hat{\epsilon}$ we have $$
|A_{t_n} \hat{h}|=O(\frac{\epsilon'}{\hat{\epsilon}^{D+\alpha_1}} + \frac{e^{-\lambda t_n}}{\epsilon'^2 \hat{\epsilon}^{2\alpha_1+2} })
 $$

Following the rest of the proof of Proposition 5.1 of Athreya, Fairchild and Masur in \cite{athreya2022counting}, on the thick part of the moduli space, $|E^k_{t_n}\cap \Lambda_{\omega}^2|$ is bound by the product of two quantities.

The first one is a counting of acceptable sector $\# I(\theta_i)$ which is shown to be $O(e^{2t}|A_t \hat{h}|)$. The second one is the maximum number of pair of saddle connections in each sector which they showed to be $O(\frac{1}{\hat{\epsilon}^{1+\delta}})$.

With the previous computation we have that
 $$
\# I(\theta_i) = e^{2t_n}|A_{t_n} \hat{h}| =e^{2t_n}O(\frac{\epsilon'}{\hat{\epsilon}^{D+\alpha_1}} + \frac{e^{-\lambda t_n}}{\epsilon'^2 \hat{\epsilon}^{2\alpha_1+2} })
$$

and then in the thick part $$
|E^k_{t_n} \cap \Lambda_{\omega}^2|= e^{2 t_n}O \left( (\frac{1}{\hat{\epsilon}})^{1+\delta}\left (\frac{\epsilon'}{\hat{\epsilon}^{D+\alpha_1}} + \frac{e^{-\lambda t_n}}{\epsilon'^2 \hat{\epsilon}^{2\alpha_1+2} } \right) \right).
$$

Grouping the thin and thick part we have $$
|E^k_{t_n}\cap \Lambda_{\omega}^2|= e^{2 t_n}O \left( \hat{\epsilon}^{1/N-2\delta}+ (\frac{1}{\hat{\epsilon}})^{1+\delta}\left (
\frac{\epsilon'}{\hat{\epsilon}^{D+\alpha_1}} + \frac{e^{-\lambda t_n}}{\epsilon'^2 \hat{\epsilon}^{2\alpha_1+2} }\right) \right).
$$
Then by Lemma 4.1 of \cite{athreya2022counting}, there is a $T_0$ such that for $t \geq T_0$,
$$
   A_t(h_A)(z,w)\pi e^{2t} \leq e^{2t} \arctan(e^{-2t}) \leq 2.
$$
Then $$
|\chi_{D_A(e^t/2,e^t)}-A_t(h_A)(z,w)\pi e^{2t}| \leq 3.
$$
As
$$
e^i_t(\omega)= (\chi_{E^i_t}\cdot (\chi_{D_A(e^t_n /2, e^t)}-\pi e^{2t}(A_t h_A)))^{SV}(\omega)
$$
and hence$$
|e^i_{t_n}(\omega)| = e^{2 t_n}O \left( \hat{\epsilon}^{1/N-2\delta}+ (\frac{1}{\hat{\epsilon}})^{1+\delta}\left (\frac{\epsilon'}{\hat{\epsilon}^{D+\alpha_1}} + \frac{e^{-\lambda t_n}}{\epsilon'^2 \hat{\epsilon}^{2\alpha_1+2} } \right) \right).
$$
So taking $\epsilon' = K e^{- \nu t}$ with $\nu= \min(\lambda/3,2)$ and as $\hat{\epsilon}^{2\alpha_1+2}$ and $\hat{\epsilon}^{D+\alpha_1} \geq \hat{\epsilon}^{D+2\alpha_1+2}$, we have for some $\lambda'$
$$
|e^i_{t_n}(\omega)| = e^{2 t_n}O \left( \hat{\epsilon}^{1/N-2\delta}+  \left (\frac{e^{-\lambda' t_n}}{\hat{\epsilon}^{D+2\alpha_1+3+\delta}}\right) \right)
$$
Finally taking $\hat{\epsilon}=e^{-f_3t}$ with $f_3=\frac{\lambda'}{D+2\alpha_1+3+\delta+1/N-2\delta}$ we have shown that for each sequence $(t_n)$ satisfying equation \ref{eq_ConditionTemp}, for almost every $\omega \in \HH$ there is a $n_0(\omega)$ such that for $n \geq n_0$, if $t_n > T_0$
 $$
|e^i_{t_n}(\omega)| = e^{2 t_n}O(e^{-f_4 t})
$$
with $f_4>0$. For technical reason, which will appear after, we choose $f_4 < \min(f,2)$.
\subsection{End of the estimation}

Putting everything together we have for every sequence $(t_n)$ respecting condition of equation \ref{eq_ConditionTemp} for almost every $\omega \in \HH$, and for every $n \geq n_0(\omega)$,
\begin{align}
  \label{eq_ComptageSuite}
  \left\vert N_A^*(\omega,e^{t_n}) -  \pi e^{2 t_n} \int_{\HH}h_A d \mu \right\vert & \leq  \left\vert N_A^*(\omega,e^{t_n}) - \pi e^{2 t_n}(A_{t_n} \hat{h}_A)(\omega) \right\vert + \left\vert \pi e^{2 t_n} \int_{\HH}h_A d \mu - \pi e^{2 t_n}(A_t \hat{h}_A)(\omega) \right\vert \nonumber \\
  & \leq \left\vert N_A^*(\omega,e^{t_n}) - \pi e^{2 t_n}(A_{t_n} \hat{h}_A)(\omega) \right\vert + |m_{t_n}(\omega)| + \sum_{i=1}^4 |e^i_{t_n}(\omega)| \nonumber \\
  & = e^{2 t_n}O(e^{-f_4 t_n})
\end{align}

We consider the family of sequence $(\frac{t_n}{2^j})$. For each sequence there is a set of full measure $\HH_i$ on which the estimate \ref{eq_ComptageSuite} is true. Considering the set $\HH_{\infty}= \cap_{j \in \mathbb{N}} \HH_i$, which is also of full measure, we have the estimate \ref{eq_ComptageSuite} for all times $\frac{t_n}{2^j}$.

Fixing a natural number $n$, let's call $K = \lfloor \frac{t_n-ln(T_0)}{ln(2)} \rfloor$ such that $\frac{e^{t_n}}{2^K} > T_0$ we have
\begin{align*}
  \frac{N_A(\omega,e^{t_n})}{e^{2 t_n}} & = \sum_{j=0}^K \frac{N_A^{*}(\omega,e^{t_n}/2^j)}{e^{2 t_n}} + \frac{N_A^{*}(\omega,e^{t_n}/2^K)}{e^{2 t_n}} \\
  & = \sum_{j=0}^K 2^{-2j}\frac{N_A^{*}(\omega,e^{t_n}/2^{2j})}{e^{2 t_n}/2^{2j}} + 2^{- 2 K}\frac{N_A^{*}(\omega,e^{t_n}/2^{2K})}{e^{2 t_n}/2^{2K}} \\
  %& = \sum_{j=0}^K 2^{-2 j} \left( \pi e^{2 t_n} \int_{\HH}h_A d \mu  e^{2 t_n}+ O(e^{-f_4 t_n}) \righ)+ 2^{-2 K}\frac{N_A^{*}(\omega,e^{t_n}/2^K)}{e^{2 t_n}/2^{2 K}}
\end{align*}

As we have $$
\frac{N_A^{*}(\omega,e^{t_n}/2^{2j})}{e^{2 t_n}/2^j} = \pi \int_{\HH} h_A d \mu + O \left( \frac{2^{f_4 j}}{e^{f_4 t_n}} \right)
$$
so
\begin{align*}
  \frac{N_A(\omega,e^{t_n})}{e^{2 t_n}} & = \sum_{j=0}^K 2^{-2j} \left( \pi \int_{\HH} h_A d \mu + O \left( \frac{2^{f_4 j}}{e^{f_4 t_n}} \right) \right) + 2^{- 2 K}\frac{N_A^{*}(\omega,e^{t_n}/2^{2K})}{e^{2 t_n}/2^K} \\
  & = 4(\frac{1-(1/4)^K}{3})\pi \int_{\HH} h_A d \mu +  O \left( \frac{1}{e^{f_4 t_n}} \right)\frac{1-2^{(f_4-2)K}}{1-2^{f_4-2}} + 2^{- 2 K}\frac{N_A^{*}(\omega,e^{t_n}/2^{2K})}{e^{2 t_n}/2^K}
\end{align*}

By proposition 3.2 and equation 1.1 of \cite{athreya2022counting} $\frac{N_A^{*}(\omega,e^{t_n}/2^{2K})}{e^{2 t_n}/2^K}$ is bounded by a constant $M$. Moreover as $f_4<2$ we have also $\frac{1-2^{(f_4-2)K}}{1-2^{f_4-2}} \leq M'$
\begin{align*}
  | \frac{N_A(\omega,e^{t_n})}{e^{2 t_n}} -c_{\mu}(A)| \leq c_{\mu}(A) (1/4)^K + \frac{f(t_n)}{e^{f_4 t_n}}M'  + 2^{- 2 K}M
\end{align*}
as $K \geq \frac{t_n-ln(T_0)}{ln(2)} - 1$ we have
\begin{align*}
  | \frac{N_A(\omega,e^{t_n})}{e^{2 t_n}} -  c_{\mu}(A) | & \leq c_{\mu}(A) (1/4)^{\frac{t_n-ln(T_0)}{ln(2)} - 1} +
   \frac{f(t_n)}{e^{f_4 t_n}}M'  + 2^{- 2 (\frac{t_n-ln(T_0)}{ln(2)} - 1)}M \\
   & = O(e^{-f_5 t_n})
\end{align*}
with $f_5= \min (2 \ln(2),f_4)$.

\subsubsection{Estimate for all times}

Let's take $t_n=\theta \log(n)$, which satisfy equation \ref{eq_ConditionTemp} for $\theta$ big enough, by monotony we have for a general time $t$ such that $n \leq t \leq n+1$ and $n \geq T_0$,

\begin{align*}
  e^{2 \theta \log(n)} \left( c_{\mu}(A) - O(e^{-f_5 \theta \log(n)}) \right) \leq N_A(\omega,e^{\theta \log(n)}) \leq N_A(\omega,e^{\theta \log(t)})  \\
  \leq N_A(\omega,e^{\theta \log(n+1)}) \leq e^{2 \theta \log(n+1)} \left( c_{\mu}(A) + O(e^{-f_5 \theta \log(n+1)}) \right)
\end{align*}

So that $$
n^{2 \theta} \left( c_{\mu}(A) - O(\frac{1}{n^{f_5 \theta}}) \right)  \leq N_A(\omega,t^{\theta}) \leq (n+1)^{2 \theta} \left( c_{\mu}(A) + O(\frac{1}{(n+1)^{f_5 \theta}})\right)
$$

Then $$
\left( \frac{n}{t} \right)^{2 \theta} \left( c_{\mu}(A) - O \left( \left( \frac{t}{n} \right)^{f_5 \theta} \frac{1}{t^{f_5 \theta}} \right) \right)  \leq \frac{N_A(\omega,t^{\theta})}{t^{2 \theta}}
\leq \left( \frac{n+1}{t} \right)^{2 \theta} \left( c_{\mu}(A) + O \left( \left( \frac{t}{n+1} \right)^{f_5 \theta} \frac{1}{t^{f_5 \theta}}  \right) \right)
$$
But then as $1-\frac{1}{t} \leq \frac{n}{t} \leq 1$ and $1 \leq \frac{n+1}{t} \leq 1+\frac{1}{t}$  we have
$$
\left\vert \frac{N_A(\omega,t^{\theta})}{t^{2 \theta}} - c_{\mu}(A) \right\vert = O\left(  \frac{1}{t^{f_5 \theta}}  + \frac{1}{t^{2 \theta}} \right)
$$

as $f_5 \leq 2 \ln(2) \leq 2$, taking $s=t^{\theta}$, we can conclude that
$$
\left| \frac{N_A(\omega,s)}{s^2} - c_{\mu}(A) \right| = O \left( \frac{1}{s^{f_5}} \right).
$$

This concludes the proof of the existence of the power saving error term in Theorem \ref{thm2}.

\section{Property of the constants $c_{\mu}(A)$}
\label{sec4}

  Fixing a $A>0$ and a ergodic $\SLR$-invariant measure $\mu$, the constant $c_{\mu}(A)$ remains mysterious.
  By the previous computation, it is equal to
 $$
 c_{\mu}(A) := \frac{4}{3} \pi \int_{\HH}\hat{h}_A d \mu = \frac{4}{3} \pi \int_{\mathbb{C}^2} h_A dm
 $$

  where $m$ is called \emph{Siegel-Veech measure} and is invariant with respect to the $\SLR$ action on $\mathbb{C}^2$. We can decompose further as
  $$
  \int_{\mathbb{C}^2} h_A dm = \int_{\mathbb{R} \setminus \{ 0\}} \left( \int_{\SLR} h(tz, w) d\lambda(z, w) \right) d \nu(t) + \int_{\mathbb{P}^1(\mathbb{R})}\left( \int_{\mathbb{C}} h(z, sz)dz \right) d \rho(s)
  $$

  where $\lambda$ is the Haar measure on $\SLR$ and $\nu=\nu(m)$ and $\rho=\rho(m)$ are decomposition of the measure $m$, on which little is known.

  We can still get a result of convergence for a $*$-weak limit of a sequence of ergodic $\SLR$-invariant probability measures.

  \begin{thm}
  Fixing an $A$, let $(\mu_n)$ be a sequence of $\SLR$-invariant probability ergodic measures on $\HH$, such that $\mu_n \to \nu$ in the weak-$*$ topology, where $\nu$ is another ergodic $\SLR$-invariant probability measure. Then the constants satisfy $$
  c_{\mu_n}(A) \to c_{\nu}(A).
  $$
\end{thm}

The proof is similar to the proof for Siegel-Veech constants made by Dozier \cite{dozier_convergence_2019}. We briefly recall it here.

\begin{proof}

  Let's call $l(X)$ the length of the shortest saddle connection on the flat surface $X$ and let $C_K=\{ X \in \HH, \frac{1}{l(X)} \leq K \}$. We take $\chi_K$ be a continuous function with image in $[0,1]$ whose value is $1$ on $C_K$ and $0$ on $\HH - C_{K+1}$.

  Then $$
  c_{\mu_n}(A) := \frac{4}{3} \pi \int_{\HH} \hat{h}_A d \mu_n = \frac{4}{3} \pi \left( \int_{\HH} \hat{h}_A \chi_K d \mu_n + \int_{\HH} \hat{h}_A (1-\chi_K) d \mu_n \right)
  $$

  We want to use $*$-weak convergence of the measure for the integral on the thick part of the stratum and to control the integral in the thin independently of the measure $\mu_n$.

  As $h_A$ is bounded and compactly supported, $\hat{h}_A$ is bounded by $\hat{f}^2$ where $f=\chi_B \| h \|_{\infty}$ with $B$ a ball big enough to contain the union of the projections of the support of $h$ via the maps $\mathbb{R}^4 \to \mathbb{R}^2$
   $$
  (x,y) \mapsto x \text{ and } (x,y) \mapsto y
  $$
 Eskin and Masur show that \cite[Theorem 5.1]{eskin_masur_2001} $\hat{f} < \frac{C}{l^{1+\delta}}$ for some $C$ and $0< \delta < 1/2$.
  Choosing a $\delta'$ in order that $\delta + \delta' < 1/2$ we obtain
  $$
\int_{\HH} \hat{h}_A (1-\chi_K) d \mu_n \leq \int_{\HH - C_K} \frac{C}{l^{1+\delta}} d \mu_n \leq \frac{C}{K^{\delta'}}\int_{\HH - C_K} \frac{1}{l^{1+\delta+\delta'}} d \mu_n
  $$
  The last integral is finite by a lemma of the two previously cited authors \cite[Lemma 5.5]{eskin_masur_2001}, and we can bound it from above by a constant independent of $\mu_n$. First we recall a bound on circle average made by Dozier.
  \begin{lem}[Proposition 1.1 of \cite{dozier_convergence_2019}]
    \label{BorneCircle}
    For any stratum $\HH$ and $0< \delta < 1/2$, there exists a function $\alpha :\HH \to \mathbb{R}_+ $ and constants $c_0,b$ such that for any $X \in \HH$, $$
    \int_0^{2 \pi} \frac{1}{l(g_t r_{\theta} X)^{1+\delta}} d\theta \leq c_0 e^{-(1-2 \delta)t}\alpha(X)+ b
    $$
  \end{lem}

  Then, choosing any smoothing function $\phi$ non-negative, smooth and compactly supported with integral over $\mathbb{R}$ equal to $1$, we get by Nevo ergodic Theorem \ref{thm_Nevo}, for a generic $X \in \HH$
  \begin{align*}
    \int_{\HH - C_K} \frac{1}{l^{1+\delta+\delta'}} d \mu_n & \leq \int_{\HH } \frac{1}{l^{1+\delta+\delta'}} d \mu_n \\
    & = \underset{t \to \infty}{\lim} \int_{-\infty}^{\infty} \phi(t-s) \left( \frac{1}{2 \pi}\int_0^{2\pi} \frac{1}{l(g_s r_{\theta}X)^{1+\delta}} d\theta \right) ds \\
    & \leq b
  \end{align*}
  So picking an $\epsilon$ we can choose $K$ big enough such that for all $n$
  $$
  \int_{\HH} \hat{h}_A (1-\chi_K) d \mu_n \leq \frac{C}{K^{\delta'}}\int_{\HH - C_K} \frac{1}{l^{1+\delta+\delta'}} d \mu_n \leq \frac{Cb}{K^{\delta'}} \leq \epsilon /2.
  $$
  Then by $*$-weak convergence of the measure there exist a $N$ such that for $n \geq N$
  $$
  \int_{\HH} \hat{h}_A \chi_K d \mu_n \leq \epsilon /2.
  $$
  Putting the two together gives the desired result.
\end{proof}

This theorem has a direct consequence on the possible values of the constants $c_{\mu}(A)$ when $\mu$ varies on the space of ergodic $\SLR$-invariant probability measure, namely they all fall in a closed interval of $\mathbb{R}_{>0}$.
Using the following proposition of Eskin, Mirzakhani, and Mohammadi

\begin{thm}[Theorem 2.3 of \cite{eskin2015isolation}]
  The space of ergodic $\SLR$-invariant probability measure is compact in the $*$-weak topology.
\end{thm}

We get the following corollary.
\begin{cor}
  For every $A$, there are two constants $m(A),M(A)$ such that for every ergodic $\SLR$-invariant measure $\mu$ $$
  m(A) \leq c_\mu(A) \leq M(A)
  $$
\end{cor}
\begin{proof}
  By contradiction taking a sequence of measure $\mu_n$ such that $c_{\mu_n}(A) \to \infty$ we can extract a subsequence $*$-weakly converging to a measure $\mu$. As the constant for this subsequence goes to $\infty$ and to the constant of the limit measure, we have a contradiction.

  The lower bound is obtained by a similar proof.
  \end{proof}

\subsection{End of proof of Theorem \ref{thm2}}

Given an ergodic $\SLR$-invariant measure $\mu$ on $\HH$, we showed in Section \ref{sec3} the quadratic growth $N_A(\omega,R)$ for $\mu$-almost every $\omega$ in $\HH$. Then in Section \ref{sec5} we proved that this quadratic growth admits a power saving error term. Finally in this Section, we demonstrated the continuity of the constants $c_{\mu}(A)$ with respect to $\mu$. By putting these results together, we get a proof of the main result, Theorem \ref{thm2}.

\bibliographystyle{plainnat}
\bibliography{bibliographie.bib}

\begin{thebibliography}{16}
\providecommand{\natexlab}[1]{#1}
\providecommand{\url}[1]{\texttt{#1}}
\expandafter\ifx\csname urlstyle\endcsname\relax
  \providecommand{\doi}[1]{doi: #1}\else
  \providecommand{\doi}{doi: \begingroup \urlstyle{rm}\Url}\fi

\bibitem[Athreya et~al.(2019)Athreya, Cheung, and Masur]{athreya2017siegel}
Jayadev~S Athreya, Yitwah Cheung, and Howard Masur.
\newblock Siegel-veech transforms are in {$L^2$}.
\newblock \emph{Journal Of Modern Dynamics}, pages 1--19, 2019.

\bibitem[Athreya et~al.(2022)Athreya, Fairchild, and
  Masur]{athreya2022counting}
Jayadev~S Athreya, Samantha Fairchild, and Howard Masur.
\newblock Counting pairs of saddle connections.
\newblock \emph{arXiv:2201.08628}, 2022.

\bibitem[Bainbridge et~al.(2018)Bainbridge, Chen, Gendron, Grushevsky, and
  Möller]{Bainbridge_2018}
Matt Bainbridge, Dawei Chen, Quentin Gendron, Samuel Grushevsky, and Martin
  Möller.
\newblock Compactification of strata of abelian differentials.
\newblock \emph{Duke Mathematical Journal}, 167\penalty0 (12), sep 2018.

\bibitem[Dozier(2019)]{dozier_convergence_2019}
Benjamin Dozier.
\newblock Convergence of {Siegel}–{Veech} constants.
\newblock \emph{Geometriae Dedicata}, 198:\penalty0 131--142, 2019.

\bibitem[Dozier(2020)]{dozier2020measure}
Benjamin Dozier.
\newblock Measure bound for translation surfaces with short saddle connections.
\newblock \emph{arXiv:2002.10026}, 2020.

\bibitem[Eskin and Masur(2001)]{eskin_masur_2001}
Alex Eskin and Howard Masur.
\newblock Asymptotic formulas on flat surfaces.
\newblock \emph{Ergodic Theory and Dynamical Systems}, 21\penalty0
  (2):\penalty0 443–478, 2001.

\bibitem[Eskin and Mirzakhani(2018)]{EskinMirzakhani}
Alex Eskin and Maryam Mirzakhani.
\newblock Invariant and stationary measures for the {$SL(2,\mathbb{R})$} action
  on moduli space.
\newblock \emph{Publications mathematiques de l'IHES}, 127\penalty0
  (1):\penalty0 95--324, 2018.

\bibitem[Eskin et~al.(2015)Eskin, Mirzakhani, and
  Mohammadi]{eskin2015isolation}
Alex Eskin, Maryam Mirzakhani, and Amir Mohammadi.
\newblock Isolation, equidistribution, and orbit closures for the {SL} (2,
  $\mathbb{R}$) action on moduli space.
\newblock \emph{Annals of Mathematics}, pages 673--721, 2015.

\bibitem[Kontsevich and Zorich(2003)]{kontsevich2003connected}
Maxim Kontsevich and Anton Zorich.
\newblock Connected components of the moduli spaces of abelian differentials
  with prescribed singularities.
\newblock \emph{Inventiones mathematicae}, 153\penalty0 (3):\penalty0 631--678,
  2003.

\bibitem[Masur(1982)]{masurmeasure}
Howard Masur.
\newblock Interval exchange transformations and measured foliations.
\newblock \emph{Annals of Mathematics}, 115\penalty0 (1):\penalty0 169--200,
  1982.

\bibitem[Masur(1990)]{masur_1990}
Howard Masur.
\newblock The growth rate of trajectories of a quadratic differential.
\newblock \emph{Ergodic Theory and Dynamical Systems}, 10\penalty0
  (1):\penalty0 151–176, 1990.

\bibitem[Masur and Smillie(1991)]{masur_smilie}
Howard Masur and John Smillie.
\newblock Hausdorff dimension of sets of nonergodic measured foliations.
\newblock \emph{Annals of Mathematics}, 134\penalty0 (3):\penalty0 455--543,
  1991.
\newblock ISSN 0003486X.

\bibitem[Nevo(2017)]{NevoErgo}
Amos Nevo.
\newblock Equidistribution in measure-preserving actions of semisimple groups :
  case of {$SL_2(\mathbb{R})$}.
\newblock \emph{arXiv preprint arXiv:1708.03886}, 2017.

\bibitem[Nevo et~al.(2020)Nevo, R{\"u}hr, and Weiss]{nevo2020effective}
Amos Nevo, Rene R{\"u}hr, and Barak Weiss.
\newblock Effective counting on translation surfaces.
\newblock \emph{Advances in Mathematics}, 360:\penalty0 106890, 2020.

\bibitem[Veech(1982)]{Veechmeasure}
William~A. Veech.
\newblock Gauss measures for transformations on the space of interval exchange
  maps.
\newblock \emph{Annals of Mathematics}, 115\penalty0 (2):\penalty0 201--242,
  1982.
\newblock ISSN 0003486X.

\bibitem[Veech(1998)]{VeechQuad}
William~A Veech.
\newblock Siegel measures.
\newblock \emph{Annals of Mathematics}, 148\penalty0 (3):\penalty0 895--944,
  1998.

\end{thebibliography}
\end{document}